\documentclass[a4paper,11pt]{article}
\usepackage{fullpage,hyperref}
\usepackage[british]{babel}
\usepackage{wrapfig}
\usepackage[comma, square, numbers]{natbib}
\usepackage{yfonts}
\usepackage[utf8]{inputenc}
\usepackage{amssymb}
\usepackage{amsthm}
\usepackage{graphics}
\usepackage{amsmath}
\usepackage{dirtytalk}
\usepackage{amstext}
\usepackage{engrec}
\usepackage{floatflt}
\usepackage{rotating}
\usepackage[safe,extra]{tipa}
\usepackage[font={footnotesize}]{caption}
\usepackage{framed}
\usepackage{listings}
\usepackage{subcaption}
\usepackage[titletoc,title]{appendix}

\newcommand{\mc}{\mathcal}
\newcommand{\mb}{\mathbb}

\newcommand{\R}{\mb R}

\newcommand{\N}{\mb N}

\newcommand{\T}{\mb T}

\newcommand{\eea}{\end{align}}

\renewcommand{\epsilon}{\varepsilon}
\renewcommand{\bar}{\overline}

\newcommand{\bo}{\bold}
\renewcommand{\phi}{\varphi}

\DeclareMathOperator{\Leb}{Leb}

\renewcommand\upsilon{\theta}

\usepackage[normalem]{ulem}

\usepackage[T1]{fontenc}		
\usepackage[utf8]{inputenc}		
\usepackage{lmodern}			
\usepackage{lastpage}			
\usepackage{indentfirst}		
\usepackage{color}				
\usepackage{graphicx}			
\usepackage{float} 				
\usepackage{chemfig,chemmacros} 
\usepackage{tikz}				
\usetikzlibrary{positioning}
\usepackage{microtype} 			
\usepackage{pdfpages}
\usepackage{makeidx}            
\usepackage{hyphenat}          
\usepackage[absolute]{textpos} 
\usepackage{eso-pic}           
\usepackage{makebox}           
\usepackage{textgreek}
\usepackage{xargs}
\usepackage{scalerel}
\usepackage{amsbsy}

\usepackage{xcolor}
\usepackage[Glenn]{fncychap}
\ChNameVar{\bfseries\Huge\rmfamily}
\ChTitleVar{\bfseries\Huge\rmfamily}
\RequirePackage{mathrsfs}
\DeclareMathAlphabet{\mathcal}{OMS}{cmsy}{m}{n}
\usepackage{mathtools}
\usepackage{amsmath}
\usepackage{amsfonts}
\usepackage{amssymb}
\usetikzlibrary{arrows}
\usepackage{tikz-cd}

\usepackage{multicol}	
\usepackage{multirow}	
\usepackage{longtable}	
\usepackage{threeparttablex}    
\usepackage{array}
\usepackage{scalerel}

\usepackage{tikz} 
\usepackage{tikz-cd}
\usepackage{amsthm}
\usepackage{tcolorbox}

\newcommandx*\din[3][1,3]{F_{\varphi_{\rule[-0.5mm]{-0.36mm}{0.21cm} \scaleobj{0.6}{#1}}, z #2}^{\scaleobj{0.7}{#3}}}

\newcommandx*{\Fvarphi}[2][1=, 2=]{%
  \mathcal{F}%
  \ifx\relax#1\relax
    \else%
      _{\rule[-0.5mm]{-0.86mm}{0.21cm} \scaleobj{0.6}{#1}}
  \fi%
  \raisebox{0.15em}{$\varphi$}
  \ifx\relax#2\relax
    \else%
      _{\rule[-0.5mm]{-0.86mm}{0.15cm} \scaleobj{0.6}{#2}}
  \fi%
}

\newcommandx*\dinfix[3][1,3]{F_{\varphi_{\rule[-0.5mm]{-0.36mm}{0.21cm} \scaleobj{0.6}{#1}}^{*}, z #2}^{\scaleobj{0.7}{#3}}}

\newcommandx*\varp[1]{\varphi_{\rule[-0.5mm]{-0.9mm}{0.31cm} \scaleobj{0.6}{#1}}}

\newcommandx*\varpp[2][1]{\varphi_{\rule[-0.5mm]{-0.9mm}{0.31cm} \scaleobj{0.6}{#1} \scaleobj{0.9}{#2}}}

\newcommandx*\varppfix[2][1]{\varphi_{\rule[-0.5mm]{-0.9mm}{0.31cm} \scaleobj{0.6}{#1} \scaleobj{0.9}{#2}}^{*}}

\newcommandx*\varpfix[1]{\varphi_{\rule[-0.5mm]{-0.9mm}{0.31cm} \scaleobj{0.6}{#1}}^{*}}

\def\W{W^{\scaleobj{0.68}{(N)}}}

\newcommandx*\diny[3][1,3]{F_{\varphi_{\rule[-0.5mm]{-0.16mm}{0.21cm} \scaleobj{0.8}{#1}}, z #2}^{\scaleobj{0.7}{#3}}}

\newcommandx*\varppp[2][1]{\varphi_{\rule[-0.5mm]{-0.1mm}{0.31cm} \scaleobj{0.8}{#1} \scaleobj{0.9}{#2}}}

\newcommand{\normadef}[2][1]{{\rule[-0.5mm]{-0.1mm}{0.28cm} \scaleobj{0.75}{#1^{#2}}}}

\newtcolorbox{HERBERTannotation}{
  colback=yellow!10!white,  
  colframe=white,           
  sharp corners             
}

\makeatletter
\newcommand{\mytag}[2]{%
  \text{#1}%
  \@bsphack
  \begingroup
    \@onelevel@sanitize\@currentlabelname
    \edef\@currentlabelname{%
      \expandafter\strip@period\@currentlabelname\relax.\relax\@@@%
    }%
    \protected@write\@auxout{}{%
      \string\newlabel{#2}{%
        {#1}%
        {\thepage}%
        {\@currentlabelname}%
        {\@currentHref}{}%
      }%
    }%
  \endgroup
  \@esphack
}

\newtheorem{theorem}{Theorem}[section]

\newtheorem{corollary}{Corollary}[section]
\newtheorem{lemma}{Lemma}[section]
\newtheorem{proposition}{Proposition}[section]
\theoremstyle{definition}
\newtheorem{definition}{Definition}[section]

\newtheorem{remark}{Remark}[section]
\newtheorem{example}{Example}[section]

\newtheoremstyle{algorithm}
{4pt}
{4pt}
{}
{}
{}
{:}
{\newline}
{}

\usepackage{xcolor}

\newtheorem{algorithm}{Algorithm}

\newcommand{\balgorithm}{\begin{algorithm}\begin{framed}\ }
\newcommand{\ealgorithm}{\end{framed}\end{algorithm}}

\newcommand{\pippo}{\begin{code} \  \begin{lstlisting}[frame=single]}
\newcommand{\pappo}{\end{lstlisting} \end{code}}

\newcommand{\bd}{\begin{definition}}
\newcommand{\ed}{\end{definition}}

\newcommand{\bt}{\begin{theorem}}
\newcommand{\et}{\end{theorem}}
\newcommand{\bp}{\begin{proposition}}
\newcommand{\ep}{\end{proposition}}
\newcommand{\bc}{\begin{corollary}}
\newcommand{\ec}{\end{corollary}} 
\newcommand{\bl}{\begin{lemma}}
\newcommand{\br}{\begin{remark}}
\newcommand{\er}{\end{remark}}

\DeclareMathOperator{\Lip}{Lip}

\usepackage{enumitem}

\usepackage{bm}
\usepackage{hyperref}

\usepackage{todonotes}

\begin{document}

\title{Selfconsistent Transfer Operators for Heterogeneous Coupled Maps}
\author{Herbert M.C. Maquera$^1$, Tiago Pereira$^{1}$, and Matteo Tanzi$^2$ \\
  \small $^1$Institute of Mathematics and Computer Sciences (ICMC), University of São Paulo, Brazil \\
  \small $^2$Department of Mathematics, King's College London, United Kingdom
}

\maketitle 

\begin{abstract}
We investigate the dynamics of large heterogeneous network dynamical systems composed of nonlocally coupled chaotic maps. We show that the mean-field limit of such systems is governed by a suitably defined Self-Consistent Transfer Operator (STO) acting on graphons, describing the infinite-size limits of dense graphs, thereby allowing for a rigorous analysis of the system as the network size tends to infinity. We construct appropriate functional spaces on which the STO has an attracting fixed point, which corresponds to the equilibrium state for the mean-field limit, and we draw a connection between the regularity properties of the graphons and the regularity of the fixed points. This work combines operator theory and graph limits tools to offer a framework for understanding emergent behavior in complex networks.
\end{abstract}

\tableofcontents

\maketitle

\section{Introduction}

Interacting dynamical systems provide models for describing a broad range of natural systems, from biology \cite{IzhikevichBook} and chemistry \cite{Kuramoto,Nijholt2022} to engineered systems \cite{RMP}. In these models, $N$ units that compose the system interact with each other according to a specific network structure. These interactions lead to rich phenomena, shaped by the network structure and the specific types of couplings between individual components \cite{PRX,chazottes2005dynamics,pecora1998master,Tanzi_2017}.

 In this paper, we focus on discrete-time coupled chaotic systems.  Interest in this system initially gained momentum in the 1980s, spurred by Kaneko’s work \cite{kaneko1984period}, which provided simulations of interesting phenomena. This work was later studied from an ergodic theoretical perspective by Sinai and Bunimovich \cite{BunSinai}. Over time, the field evolved to focus on two primary cases: coupled maps lattices, where nodes interact with their nearest neighbors in a regular lattice \cite{Keller1, Keller2, chazottes2005dynamics}, and globally coupled maps, where every node interacts with all others \cite{selley2016symmetry}.

Recent research has moved towards a more generalized framework of interactions, specifically where the interaction network is heterogeneous \cite{Koiller-Young,Matteo}. In such systems, nodes receive differing numbers of interactions, resulting in ``heterogeneously coupled'' systems with widely varying connectivity degrees across nodes.

In heterogeneous networks with multiple layers of connectivity, distinct local mean fields emerge depending on a node's degree \cite{bian2025mean}, and certain groups of nodes can exhibit synchronous behaviour while others behave erratically \cite{corder2023emergence}. Extending this observation to the thermodynamic limit presents important challenges. Most notably, the heterogeneous structure of the network requires handling multiple local mean fields simultaneously.

In this work, we investigate heterogeneous networks of coupled expanding maps. In the thermodynamic limit, we describe  the mathematical construct of a \textit{graphon}, which is a measurable, symmetric function $W: [0,1] \times [0,1] \to \mathbb{R}$ that represents the adjacency matrix. Using graphons to study large limits of networked continuous time dynamical systems is a recent development that has provided multiple dynamical insights for large networks, for example in Vlasov and  MacKean-Vlasov  equations, and in generalisation of the Kuramoto (see \cite{Kuehn2020,carrillo2020long,bertoli2025phase, coppini2022note, jabin2025mean} and references therein). In these setups graphons enable the transformation of finite networks into continuous systems, facilitating the study of synchronization and clustering \cite{Bick2021}:  Invariant subspaces generated by graphon automorphisms give information on partially synchronized states, while block-structured graphons model diverse inter-population couplings for heterogeneous networks, providing a foundation for detailed analyses of differentiated cluster dynamics.

In a related extension of the Kuramoto model, the authors of \cite{Bick_2023} introduce higher-order interactions in phase oscillator networks. Using graphon theory, they analyse these networks in the infinite limit, showing how triplet and quadruplet interactions lead to new stability behaviours and bifurcations as coupling parameters vary. Their model reveals that higher-order interactions significantly enrich the dynamics of the network.

Here, we focus on the ergodic theory aspects of networks of coupled discrete time chaotic dynamical systems in the limit of large network size. We define our Selfconsistent Transfer Operator (STO) as the transfer operator of the map \( F_{\nu}: [0,1] \times \mathbb{T} \to [0,1] \times \mathbb{T} \), which fixes fibers in the first coordinate. In contrast, the second coordinate is determined by the dynamics within the fiber in the thermodynamic limit. We consider absolutely continuous measures with respect to the Lebesgue measure with density \( \varphi \). Thus, instead of studying the evolution of measures, we focus on the evolution of densities.

While previous works have developed selfconsistent transfer operator frameworks for globally coupled maps and homogeneous networks \cite{tanzi2023mean,selley2021linear}, the challenge of capturing heterogeneous, structured interactions in the thermodynamic limit remains open. This paper addresses that gap by introducing a functional analytic framework for STOs on graphon-based network limits, which allows us to determine the existence, uniqueness, and stability of fixed points for the STO under a small coupling strength assumption. 
To analyze the STO, we introduce a function space informed by Keller's framework \cite{KellerPV} for generalizing functions of bounded variation to a quasi-compact pseudometric space equipped with a finite Borel measure.

Additionally, we demonstrate that contraction occurs in almost every fiber. This is achieved by estimating two terms via the classical triangle inequality. For the first term, we leverage the fact that dynamics on the fiber \( \{z\} \times \mathbb{T} \) are expanding maps. This allows us to apply results on uniformly expanding maps acting on cones, establishing a consistent contraction property.

\section{Model and Setup}

\subsection{Heterogeneously coupled maps}

\paragraph{Local dynamics} Let $\mathbb{T} = \mathbb{R}\backslash \mathbb{Z}$ and $f\in \mathcal{C}^{3}(\mathbb{T},\mathbb{T})$ be a uniformly expanding circle map, that is, there exists $\sigma >1$ such that $|f'(x)| \ge\sigma$ for all $x\in \mathbb{T}$. Let's denote by $d>1$ the degree of $f$. Let $\hat f$ be a lift of $f$ and consider a collection of $N$ dynamical systems 
$$
x_i(t+1) = \hat f(x_i(t)) \,\, \text{mod 1}, \mbox{~for ~} i = 1, 2, \dots, N.
$$
We refer to this as the local dynamics. Abusing notation  $f$ will denote the map and its lift whenever there is no ambiguity. 

\paragraph{Coupling} Let $h\in \mathcal{C}^{3}(\mathbb{T}\times\mathbb{T},\mathbb{R})$ be the function that describes the interaction between pairs of interacting units, and $\alpha\in\mathbb{R}$ a parameter standing for the coupling strength. We are interested in the dynamics of such expanding maps when they interact on a graph $G$, namely: \begin{equation}\label{Eq:CoupledDynamics}
    x_{i}(t+1)=f(x_{i}(t))+\dfrac{\alpha}{N}\sum_{j=1}^{N}A_{ij}h(x_{i}(t),x_{j}(t))\, \text{~ mod 1, } \mbox{~for ~} i = 1, 2, \dots, N
\end{equation}
where $A = (A_{ij})_{i,j=1}^N$ is the adjacency matrix of a graph $G$ defined as $A_{ij}=1$ if $i$ receives a connection from $j$ and zero otherwise. 
Equation \eqref{Eq:CoupledDynamics} defines a map $F:\mathbb{T}^{N}\to \mathbb{T}^N$ such that the components $F_i(x_1(t),...,x_N(t))=x_i(t+1)$ are given by Eq. (\ref{Eq:CoupledDynamics}).
 In the limit of large $N$, the statistical properties of this class of models have been thoroughly studied when the graph $G$ is all-to-all connected. We are interested in the case when the graph is heterogeneous, that is, the number of connections the $i$th node receives $k_i = \sum_{j=1}^N A_{ij}$, also called degree of the node $i$, varies among the nodes. We will restrict ourselves to dense graphs and study the statistical properties of such heterogeneous systems within the limit of large $N$ via graphons.

\subsection{Self-consistent transfer operators on graphons}

\begin{definition}
A \emph{graphon}  \( (W, \mathcal{S}) \), where \( \mathcal{S} = (A, \mathcal{A}, \mu) \) is a \(\sigma\)-finite measure space with \( \mu(A) > 0 \), and \( W : A \times A \to \mathbb{R} \), is a  measurable function such that \( W \in L^1(A \times A, \mu \times \mu) \). That is, \( W \) is measurable with respect to the product \(\sigma\)-algebra \( \mathcal{A} \otimes \mathcal{A} \), and integrable with respect to the product measure \( \mu \times \mu \). We refer to \( W \) as a graphon over \( \mathcal{S} \).
\end{definition}

For simplicity, we are going to restrict to the case where $\mc S$ is $([0,1],\mc B([0,1]),\Leb)$.
\newcommand\restr[2]{{
  \left.\kern-\nulldelimiterspace 
  #1 
  \vphantom{\big|} 
  \right|_{#2} 
  }}
A graph $G^{(N)}$ on $N$ vertices\footnote{For us graphs will be always defined on the set of vertices $\{1,..., N\}$ with directed edge set contained in $\{1,..., N\}^2$ and associated $N\times N$ adjacency matrix $(A_{ij})$. }  can be represented as a graphon on $[0,1]$ in the following standard way: divide this interval into $N$ subintervals of equal length that are open on the left, closed on the right -- except the first interval $[0,1/N]$ that is closed: 
\[
I_1:=[0,1/N],\quad I_i:=\left(\frac{i-1}{N},\frac iN\right]\quad\quad i=2,...,N;
\] 
define $\W$ by
\begin{equation}\label{Eq:FiniteSizeGraphon}
\restr{\W}{I_i\times I_j}(x,y):=A_{ij}.
\end{equation}
With this definition, it holds that for any $z\in I_i$ 
\begin{equation*}
    \dfrac{1}{N}A_{ij}=\ \int_{\frac{j-1}{N}}^{\frac{j}{N}} dz' \W(z,z').
\end{equation*}
\noindent
It is not hard to show that \eqref{Eq:CoupledDynamics} can be rewritten as:
\begin{equation}\label{Eq:Frefgraphon}
    x_{i}(t+1)=f_{i}(x_{i}(t))+\alpha \sum_{j=1}^{N} \int_{\frac{j-1}{N}}^{\frac{j}{N}} dz' \, \W(z,z') h(x_{i}(t),x_{j}(t))
\end{equation}
for any $z\in I_i$. 
{Graphons arise as natural limit objects enabling a rigorous passage from discrete network interactions to a continuum framework that preserves the heterogeneous structure of the original system; in particular, the convergence $W^{(N)} \to W$ (in an appropriate sense)  justifies replacing discrete adjacency sums with integrals against the limiting graphon, allowing the derivation of the limit equations, which we now introduce.
}

\paragraph{Self-Consistent transfer operators on graphons} 
Let  $\mc M([0,1]\times \T)$ be the set of signed finite measures on $[0,1]\times \T$. Let $\mc M_{\Leb}([0,1]\times \T)\subset \mc M([0,1]\times \T)$ be the subset of measures with marginals on the first coordinate equal to $\Leb$, and $\mc M_{1,\Leb}([0,1]\times \T)\subset \mc M_{\Leb}([0,1]\times \T)$ the probability measures. For $\nu\in \mc M_{1,\Leb}([0,1]\times \T)$, denote by $\{\nu_z\}_{z\in[0,1]}$ its disintegration with respect to the measurable partition $\{\{z\}\times \T\}_{z\in[0,1]}$.  Then, $\{\nu_z\}_{z\in [0,1]}$ satisfies:
\begin{itemize}
\item $z\mapsto \nu_z(A)$ is measurable for every measurable set $A\subset [0,1]\times\T$;
\item $\nu(A)=\int_0^1\nu_z(A)d\Leb(z)$.
\end{itemize}
With an abuse of notation, we identify fibers $\{z\}\times \T $ with $\T$, and $\nu_z$ as measures on $\T$.

\begin{definition}[Self-Consistent Transfer Operator on a Graphon $W$] Given $f \in \mathcal{C}^3(\mathbb{T},\mathbb{T})$ the local dynamics,  $h\in  \mathcal{C}^3(\mathbb{T}\times \mathbb{T},\mathbb{R})$ the coupling function,  and   $W\in L^\infty([0,1],L^1([0,1]))$\footnote{By which we mean that $W(z,\cdot)$ is in $L^1([0,1],\R)$ for almost any $z\in[0,1]$ and the essential supremum of $\|W(z,\cdot)\|_{L^1}$ is bounded.} a graphon, for any measure $\nu\in \mathcal{M}_{1,\Leb}([0,1]\times \mathbb{T})$ and $\Leb$-a.e. $z\in [0,1]$ we define 
\begin{itemize}
    \item $F_{\nu}:[0,1]\times\mathbb{T}\rightarrow [0,1]\times \mathbb{T}$, where $F_{\nu}(z,x)=(z,F_{\nu,z}(x))$  where 
    \begin{equation}\label{Eq:SetDynFiber}
    F_{\nu,z}(x)=f(x)+\alpha\int_{0}^{1}dz' \int_{\mathbb{T}}\ \ d\nu_{z'} (y)\, W(z,z') \, h(x,y)
    \end{equation}
   
    \item $\mathcal{F}:\mathcal{M}_{1,\Leb}([0,1]\times\mathbb{T})\to \mathcal{M}_{1,\Leb}([0,1]\times\mathbb{T})$, given $\nu \in \mathcal{M}([0,1]\times\mathbb{T})$ then
    \begin{equation}\label{Eq:STODef}\mathcal{F}\nu=(F_{\nu})_{*}\nu
    \end{equation}
    where $(F_{\nu})_{*}$ denotes the push-forward under $F_\nu$.
\end{itemize}
\end{definition}
\begin{remark}
It is not hard to show that the above definitions are well posed and that denoting $\{(\mc F\nu)_z\}_{z\in[0,1]}$ the disintegration of $\mc F\nu$ w.r.t. the usual foliation, one has:
 \begin{equation}
(\mathcal{F}\nu)_z=(F_{\nu,z})_*\nu_z.
\end{equation}
\end{remark}

Our main results in this paper concern: justifying $\mc F$ as an object describing the $N\rightarrow \infty$ limit for the sequence of finite-dimensional systems given in \eqref{Eq:CoupledDynamics}; showing that under suitable conditions $\mc F$ has a unique, locally attracting fixed absolutely continuous probability measure. Before introducing our main results, we present some of the functional spaces that we will use.

\subsection{Weak and Strong Function Spaces}

\paragraph{ Wasserstein Distance} Let $\mathsf{M}$ be a compact metric space and denote by $\mathcal{M}_1(\mathsf{M})$ the set of probability measures on $\mathsf{M}$. For $g:\mathsf{M} \to \mathbb{R}$ we consider the distance between measures given by 
\begin{equation}
W^{1}(\mu,\nu) = \sup_{
\substack{\Lip(g)\leq 1 \\ \Vert g \Vert_{\infty}\leq 1 } } \int g \cdot d\mu - \int g \cdot d\nu 
\end{equation}

\noindent
We define 
\begin{equation*}
\Vert \mu \Vert_{W^1}:= W^{1}(0,\mu).
\end{equation*}
and note that $\Vert \cdot \Vert_{W^1}$ is a norm on the space of signed measures.

\paragraph{Disintegrations with Lebesgue Marginal} 
For a measure $\nu\in \mc M_{1,\Leb}$, has disintegration $\{\nu_z\}_{z\in[0,1]}$, where $\nu_z$ a $\Leb$-almost surely a probability measure.
We will mostly be concerned with absolutely continuous measures on $[0,1]\times \T$. Let 
$\varphi\in L^1([0,1]\times\mathbb{T}, \mathbb{R})$ be a probability density, and for convenience, introduce the notation 
$$
\varphi_z(x) : = \varphi(z,x).
$$
We introduce functional spaces useful to control the regularity of $\phi$ in the first and second variables.
We denote by $\|\phi_z \|_{L^1}$ the $L^1(\T)$ norm of $\phi_z$, and recall that it can be expressed as 
\begin{equation}\label{Eq:L^1Norm}
\|\phi_z \|_{L^1}=\sup_{\substack{g\in \mc C^0(\T,\R)\\ \Vert g\Vert_{\infty}\le 1}}\int_\T g(s)\phi_{z}(s) ds.
\end{equation}
\paragraph{Bounded Variation Seminorms of Densities on $\T$}
We will need two seminorms to control the smoothness of the densities in $L^1(\T)$ by bounding the variation (i.e., BV seminorms) and the variation of the derivative. For $\psi\in L^1(\T)$,  define: seminorms
\begin{align*}
|\psi|_{BV^1}&:=\sup_{\substack{g\in \mc C^1(\T,\R)\\ \Vert g\Vert_\infty\le 1}}\int_\T g'(s)\psi(s)ds\\
|\psi|_{BV^2}&:=\sup_{\substack{g\in \mc C^2(\T,\R)\\ \Vert g\Vert_\infty\le 1}}\int_\T g''(s)\psi(s)ds;
\end{align*}
  norms $\|\psi\|_{BV^1}:=|\psi|_{BV^1}+\|\psi\|_{L^1}$ and $\|\psi\|_{ {BV^2}}:=|\psi|_{BV^2}+\|\psi\|_{L^1}$; and spaces 
\begin{align*}
\mc B_{{BV^1}}&:=\{\psi\in L^1(\T):\,\|\psi\|_{BV^1}<\infty\}\\
\mc B_{BV^2}&:=\{\psi\in L^1(\T):\,\|\psi\|_{BV^2}<\infty\}.
\end{align*}
For $M\ge 0$ denote by $\mc B_{{BV^i},M}$ the ball of radius $M$ centered at zero in $\mc B_{{BV^i}}$. 

{
\paragraph{Fiberwise Regularity and Variation Control.}
To describe the regularity of densities of measures in $\mc M_{1,\Leb}([0,1],\T)$ along the fiber direction, we introduce the sets
\begin{equation}\label{Set:RegOnFib}
\tilde {\mc B}_{BV^i,M}:=\{\phi\in L^1([0,1]\times \T,\R):\, \phi_z\in \mc B_{BV^i,M}\mbox{ for a.e. } z\in[0,1]\}, \,\text{for} \ \ i=1,2.
\end{equation}
These sets impose uniform control of the regularity of the function $\phi_{z}$ along each fixed fiber $\{z\} \times \mathbb{T}$. Specifically, they ensure that for almost every $z$, the function $x \mapsto \phi_z(x) $ belongs to the space of functions of bounded variation of order $i$, with norm bounded by $M$. The role of these sets is to ensure uniform control of the smoothness of densities in the fiber direction, i.e., in the variable $x \in \mathbb{T}$, independently of the base parameter $z \in [0,1]$. This fiberwise control plays a crucial role in the construction of invariant sets for the STO, allowing the application of compactness and stability arguments in the vertical direction.

}
\paragraph{Bounded Variation of the Disintegration.} To apply Schauder's fixed point theorem and prove the existence of a fixed point, we need some compactness property of the function spaces invariant under the action of the STO. We need to keep track of the regularity of the disintegrations with respect to the first variable $z\in [0,1]$. Such spaces have been used before in \cite{Galatolo2018} (see \cite{GiuMarTanzi} where the regularity of the disintegration plays a role in determining the invariant measures of the system). 
Suppose that $\varphi\in L^1([0,1],L^1(\T))$\footnote{We could have simply written $\phi\in L^1([0,1]\times \T)$, but we use this notation to emphasise that we are thinking of $\phi(z,x)$ as a mapping $z\mapsto\phi_z(x)$.}, then define the following function
\begin{equation}
\mathtt{osc}_{BV^1}(\varphi,\omega,r):=\mathop{\mathrm{ess~sup}}\limits_{z,\bar{z}\in B(\omega,r)} | \varphi_{\bar{z}}-\varphi_{z}|_{BV^1}
\end{equation}
where $B(\omega,r)$ is the ball with radius $r>0$ and centered in $\omega$, the essential supremum is with respect to  Lebesgue measure $\Leb$ on $[0,1]$, and we used the notation $\phi_z(x):=\phi(z,x)$. The function $\omega\mapsto\mathtt{osc}_{BV^1}(\varphi,\omega,r)$ is lower semi-continuous, and therefore measurable. Thus, we consider the integral 
\begin{equation}
\mathtt{var}_{p,BV^1}(\varphi,r)=\dfrac{1}{r^{p}}\int_0^1\mathtt{osc}_{BV^1}(\varphi,\omega,r) \ d\omega
\end{equation}
\noindent
and take the supremum
\begin{equation}\label{Eq:VarpNorm}
\mathtt{var}_{p,BV^1}(\varphi)=\sup_{r>0} \mathtt{var}_{p,BV^1}(\varphi,r)
\end{equation}

Define the weak norm 
\[
\Vert \varphi \Vert_{``1"}: =\int_0^1\Vert \varphi_{z} \Vert_{W^1} dz
\]
and the weak Banach space (see Definition 14 in \cite{Galatolo2018})
$$ \mathcal{B}^{}_{w}:= \left\lbrace \varphi:[0,1]\times\mathbb{T}\to \mathbb{R}:  \Vert \varphi \Vert_{``1"}< \infty\right\rbrace
$$
and the strong norm 
\[
\Vert \varphi \Vert_{\textit{S}} :=\mathtt{var}_{p,BV^1}(\varphi)+\Vert \varphi \Vert_{``1"}
\]
 and a strong Banach space
$$ \mathcal{B}_{\textit{S}}:= \left\lbrace \varphi:[0,1]\times\mathbb{T}\to \mathbb{R}: \|\varphi\|_{\textit{S}}  < \infty\right\rbrace.$$
We denote by $\mc B_{\textit{S,M}}\subset\mc B_{\textit{S}}$ the ball of $\mc B_{\textit{S}}$ centered at zero with radius $M>0$. 

{
\paragraph{Admissible Set of Regular Densities}
To construct a suitable domain for the application of Schauder's fixed-point theorem, we define the set 
\begin{equation}\label{Eq:SpaceofMeasures}
\mc A_{\bo M}:=\mc B_{\textit{S,M}}\cap \tilde {\mc B}_{BV^1,M_1}\cap \tilde{\mc B}_{BV^2,M_2}\cap \mc M_{1,\Leb}
\end{equation}
where $\bo M=(M_1,M_2,M)$. This set consists of probability densities \( \varphi \) on \( [0,1] \times \mathbb{T} \) that satisfy:

\begin{itemize}
    \item[(i)] \textit{Basewise regularity}: \( \varphi \) lies in a ball of radius \( M \) in the strong Banach space \( \mathcal{B}_{\textit{S}} \), so the map \( z \mapsto \varphi_z \) has controlled variation in the \( \mathrm{BV}^1 \)-seminorm.
    
    \item[(ii)] \textit{Fiberwise regularity}: for almost every \( z \in [0,1] \), the \( \varphi_z \) has controlled generalised first and second derivatives, as it lies in both \( \mathcal{B}_{\mathrm{BV}^1} \) and \( \mathcal{B}_{\mathrm{BV}^2} \), with uniform bounds \( M_1 \) and \( M_2 \);
    
    \item[(iii)] \textit{Marginal constraint}: the marginal of \( \varphi \) in the base coordinate is the Lebesgue measure;
\end{itemize}
These conditions ensure that \( \mathcal{A}_{\bo M} \) forms a suitable functional framework for applying fixed point arguments, leading to the existence of statistically stable states in heterogeneous systems.
}

\paragraph{Definition variation of graphon}
Let $W:[0,1]\times[0,1]\to\R$ be a graphon; if for a.e.\ $z\in[0,1]$ the section $W(z,\cdot)\in L^1([0,1])$, we fix $p\in(0,1]$ and for $\omega\in[0,1]$ and $r>0$ denote by $B(\omega,r)\subset[0,1]$ the ball of radius $r$ centered at $\omega$. We define the local $L^1$ oscillation as
\[
\mathtt{osc}_{L^1}(W,\omega,r)
:=\operatorname*{ess\,sup}_{z,z'\in B(\omega,r)} 
   \|W(z,\cdot)-W(z',\cdot)\|_{L^1([0,1])},
\]
and the $p$--variation in $L^1$ as
\[
\mathtt{var}_{p,L^1}(W,r):=\frac{1}{r^{p}}\int_0^1 \mathtt{osc}_{L^1}(W,\omega,r)\,d\omega,
\qquad 
\mathtt{var}_{p,L^1}(W):=\sup_{r>0}\mathtt{var}_{p,L^1}(W,r).
\]

Several standard classes of graphons satisfy $\mathtt{var}_{p,L^1}(W)<\infty$. 
This holds, for instance, when $z\mapsto W(z,\cdot)$ is H\"older or Lipschitz in $L^1$, since local oscillations scale at most like $|z-z'|^\alpha$; for uniformly continuous graphons with modulus $\omega_W(r)=O(r^p)$; for piecewise-regular or block graphons, where the contribution of discontinuities is controlled by the finite number of jump points; for spatially decaying kernels of the form $W(z,y)=\xi(|z-y|)$ with $\xi$ Lipschitz; and trivially for constant (Erd\H{o}s--R\'enyi) graphons, where the variation vanishes identically.

\section{Main Results}

Our work differs from previous graphon results as we focus on deterministic chaotic maps, in constructing a rigorous STO in a heterogeneous setting.  We establish not only the existence of a thermodynamic limit but also the convergence and stability. Unlike prior STO results  \cite{liverani2019statisticalpropertiesuniformlyhyperbolic, selley2016symmetry}, under further assumption, we show that empirical node statistics are also described by the thermodynamic limit operator, establishing the usefulness of the thermodynamic limit in the finite system size. We start with this result.

\subsection{Relationship between the finite-dimensional system and the STO}

Our next result provides the essential link between the abstract infinite-dimensional dynamics given by the STO and the observable statistical evolution of the finite network dynamical system. In particular, we are going to show that under suitable hypotheses, the STO captures the large $N$ evolution of marginals of measures on $\T^N$ on single coordinates.

Suppose that $\nu\in \mc M_{1,\Leb}([0,1]\times\T)$ with a disintegration $\{\nu_z\}_{z\in [0,1]}$. For every $N\in \N$ pick coordinates at each node distributed according to a triangular array of random variables $\{x_i^{(N)}\}_{N=1,i=1}^{\infty,N}$\footnote{Jointly defined on some unspecfied probability space $(\Omega,\mb P)$.} where we denote by $\bo\mu^{(N)}$ the joint distribution of $(x_1^{(N)},...,x_N^{(N)})\in \T^N$ which is such that $x_i^{(N)}$ is distributed according to the probability measure $\mu_i^{(N)}$ defined as
\[
\mu_i^{(N)}(A):=N\int_{\frac{i-1}{N}}^{\frac{i}{N}}\nu_z(A)dz,
\]
i.e  the marginal of $\mu^{(N)}$ on the $i$-th coordinate is given by $\mu_i^{(N)}$ as define above: $\Pi_i\mu^{(N)}=\mu_i^{(N)}$, where $\Pi_i=(\pi_i)_*$ and $\pi_i:\T^N\rightarrow \T$ is defined as $\pi_i(x_1,...,x_N)=x_i$. 
Notice that with this choice, the measure  $\nu$ naturally prescribes the distribution at each node in the limit $N\rightarrow\infty$ in the sense that for any $z_*\in[0,1]$ at which the function $z\mapsto \nu_z$ is continuous (with respect to the Wasserstein metric on measures),
$
\lim_{N\rightarrow\infty}\Pi_{\lceil z_*N\rceil}\mu^{(N)}=\nu_{z_*}
$ weakly.

Moreover, define

$$(y_1^{(N)},...,y_N^{(N)}):=F(x_1^{(N)},...,x_N^{(N)}),$$ 

with $F$ as  in \eqref{Eq:CoupledDynamics} (recall also the reformulation of $F$ via a graphon $W^{(N)}$ in \eqref{Eq:Frefgraphon}), and notice that  $y_i^{(N)}$ is distributed according to $\Pi_iF_*\mu^{(N)}$. The main result of this section states that $\Pi_{\lceil z_*N\rceil}F_*\mu^{(N)}$ converges weakly to $(\mc F\nu)_{z_*}$, provided that suitable assumptions on the concentration of $\mu^{(N)}$ and on the convergence of $\W$ are satisfied.

\paragraph{Hypothesis (H1) on concentration properties of $\mu^{(N)}$} We assume that the sequence $\bo \mu^{(N)}$ satisfies the following concentration inequalities: for any $C>0$, there are $C_1,C_2>0$  such that for any $N\in \N$ and any function $\psi(x_1,...,x_N)$ on $\T^N$ with values in $\R$ that is $CN^{-1}$-Lipschitz in all of its entries\footnote{This means that for any $1\le i\le N$ and any $x_1, x_2,...,x_N,x_i'\in\T$
\[
|\psi(x_1,...,x_i,...,x_N)-\psi(x_1,...,x_i',...,x_N)|\le CN^{-1}|x_i-x_i'|.
\]} the following holds: 
\begin{equation}\label{Eq:H1}
\bo \mu^{(N)}\left(\left\{x\in \T^N:\,\left|\psi(x)-\mb E_{\mu^{(N)}}[\psi]\right|>\epsilon\right\}\right)\le C_1\exp[-C_2\epsilon^2N] \quad\forall\epsilon>0.
\end{equation}

\begin{remark}\label{Rem:HigherIterates}
(H1) is satisfied if, for example, $\mu^{(N)}$ is a product measure on $\T^N$, i.e. $\mu^{(N)}=\mu_1^{(N)}\otimes...\otimes\mu_N^{(N)}$, as a consequence of McDiarmid's concentration inequality. Notice that in the case where $\psi$ is the average of a Lipschitz function, \eqref{Eq:H1} is implied by Hoeffding's inequality. 

It can also be easily inferred that if the hypothesis (H1) is satisfied by $\mu^{(N)}$, it is also satisfied by the pushed-forward measure $F_*\mu^{(N)}$\footnote{Suppose that $\mu^{(N)}$ satisfies hypothesis (H1), then for any fixed $C>0$ one has that: for every $N\in \N$  and any $\psi:\T^N\rightarrow \R$ that is $CN^{-1}$-Lipschitz in every coordinate \begin{align*}
F_*\mu^{(N)}\left(\left|\psi(x_1,...,x_N)-\mb E_{F_*\mu^{(N)}}[\psi]\right|>\epsilon\right)&=\mu^{(N)}\left(\left|\psi\circ F(x_1,...,x_N)-\mb E_{\mu^{(N)}}[\psi\circ F]\right|>\epsilon\right)\\
&\le \tilde C_1\exp(-\tilde C_2\epsilon^2N)
\end{align*} 
for some $\tilde C_1,\,\tilde C_2>0$ depending on $C>0$ only. To prove the inequality above, one only needs to observe that $\tilde \psi(x_1,...,x_N)=\psi\circ F(x_1,...,x_N)$ is $\tilde CN^{-1}$-Lipschitz in every coordinate for $\tilde C:= C[\|f\|_{\mc C^1}+2|\alpha|\|h\|_{\mc C^1}]$ -- in fact letting $(y_1,...,y_N)=F(x_1,...,x_i,...,x_N)$  and $(y_1',...,y_N')=F(x_1,...,x_i',...,x_N)$ one has $|y_i-y_i'|\le \|f\|_{\mc C^1}+|\alpha|\|h\|_{\mc C^1}$ and for $j\neq i$, $|y_j-y_j'|\le N^{-1}|\alpha|\|h\|_{\mc C^1}$, and therefore, by the Lipschitz property of $\psi$ and triangle inequality, 
\[
|\psi(y_1,...,y_N)-\psi(y_1',...,y_N')|\le CN^{-1}[\|f\|_{\mc C^1}+|\alpha|\|h\|_{\mc C^1}]|x_i-x_i'|+\sum_{j\neq i} CN^{-1}N^{-1}|\alpha|\|h\|_{\mc C^1}|x_i-x_i'|\le \tilde CN^{-1} |x_i-x_i'| .
\]} and therefore by $F_*^n\mu^{(N)}$ for any finite $n\in \N$.
\end{remark}

\paragraph{Hypothesis (H2) on the Convergence of the Graphon} Suppose $W^{(N)}\in L^\infty([0,1],L^1([0,1]))$ is a sequence of graphons. We are going to work under the following convergence hypothesis: there is $E\subset[0,1]$ and $W\in L^\infty([0,1],L^1([0,1]))$ such that
\begin{equation}\tag{H2}\label{Eq:CH}
\lim_{N\rightarrow \infty}\int_{0}^{1} dz \, \W(z_*,z) \int_\T h(x,y)d\nu_{z}(y) = \int_{0}^{1} dz \, W(z_*,z) \int_\T h(x,y)d\nu_{z}(y)\quad\quad\forall x\in\T,\,\forall z_*\in E.
\end{equation}

\begin{remark}\label{Rem:HypH1}
As stated more precisely in Lemma \ref{Lem:H2Suf11} and Lemma \ref{Lem:H2Suf12} below, sufficient conditions for (H2) to hold are: convergence in $L^1$ of the sequences $W^{(N)}(z_*,\cdot)$ for every $z_*\in E$; or 2) weak convergence of $W^{(N)}(z_*,\cdot)$ and global Lipschitz property of $z\mapsto \nu_z$. 1) and 2) allow us to treat different examples of interest in applications.
\end{remark}

\begin{theorem}\label{ConvergEmpiricalMeasure} Fix $\nu$ and assume that for $z_*\in E\subset [0,1]$ there is $W\in L^\infty([0,1],L^1([0,1]))$ such that the convergence hypothesis (H2) holds. Furthermore assume that  $z\mapsto \nu_z$ is Lipschitz in a neighbourhood of each $z_*\in E$. 
Then, for a sequence of measures $\{\mu^{(N)}\}_{N\in\N}$ and satisfying hypothesis (H1), we have

\[
\lim_{N\rightarrow \infty}\Pi_{\lceil z_*N\rceil}F_*\mu^{(N)}=(\mc F\nu)_{z_*}
\]

\noindent
weakly for all $z_*\in E$.
\end{theorem}

\begin{example}[Clustered Network]
Suppose that a network on $N$ nodes is divided into two groups: group one made of nodes $\{1,...,\lfloor N/2\rfloor\}$ and group two made of nodes $\{\lfloor N/2\rfloor+1,...,N\}$. Suppose also that the nodes in group one interact all-to-all among each other with coupling strength $\alpha_1$, and nodes in group two interact all-to-all among each other with coupling strength $\alpha_2$. The nodes interact across groups with strength $\alpha_3$. A graphon representation of these interactions is given by the piecewise constant $\W$:
\[
\W(z,z')=\left\{
\begin{array}{ll}
\alpha_1 & (z,z')\in [0,\frac{\lfloor N/2\rfloor}{N}]\times [0,\frac{\lfloor N/2\rfloor}{N}]\\
\alpha_2 & (z,z')\in (\frac{\lfloor N/2\rfloor}{N},1]\times (\frac{\lfloor N/2\rfloor}{N},1]\\
\alpha_3 & \mbox{otherwise}.
\end{array}
\right.
\]
It is not hard to show that for almost every $z\in[0,1]$ (in fact, for every $z\neq 1/2$), $\W(z,\cdot)\rightarrow W(z,\cdot)$ in $L^1$ where
\[
W(z,z')=\left\{
\begin{array}{ll}
\alpha_1 & (z,z')\in [0,1/2]\times [0,1/2]\\
\alpha_2 & (z,z')\in (1/2,1]\times (1/2,1]\\
\alpha_3 & \mbox{otherwise}.
\end{array}\right.
\]
Picking a measure $\nu\in \mc M_{1,\Leb}$, we want to account for the fact that the two distinct groups might be following different mean-field dynamics, or be initialised according to different initial conditions,  and therefore present different orbit distributions. So we should require $\nu_z=\nu_1$ for $z\in[0,1/2]$ and $\nu_z=\nu_2$ for $z\in (1/2,1]$ with $\nu_1,\,\nu_2$ two probability measures on $\T$ (notice that iterates of $\nu$ under the STO associated to $W$ preserve this property).  Now one can show that this choice of $\W$ and $\nu$ satisfies hypothesis (H2) with $E=[0,1/2)\cup(1/2.1]$ as a consequence of Lemma \ref{Lem:H2Suf11} below. 

This example can be generalised to networks with a higher number of clusters.
\end{example}

{\begin{example}[Network with Spatially Decaying Interactions] Fix $\xi:[0,1]\rightarrow\R$ a Lipschitz function. Consider a network on $N$ nodes such that the coupling strength between node $i$ and $j$ is given by $\xi(\frac{|i-j|}{N})$. A graphon representation of this network is given by 
\[
\W(z,z')=\xi\left(\frac{|i-j|}{N}\right)\mbox{ for }(z,z')\in \left(\frac{i-1}{N},\frac{i}{N}\right]\times\left(\frac{j-1}{N},\frac{j}{N}\right].
\]
One can show that $\W(z,\cdot)$ converges to $W(z,\cdot)=\xi(|z-\cdot|)$ for any $z\in[0,1]$ uniformly and thus in $L^1$. Pick any $\nu$ with $z\mapsto \nu_z$ being Lipschitz. Applying Lemma \ref{Lem:H2Suf11}, we can see that hypothesis (H2) is satisfied with $E=[0,1]$.

This setting is related to coupled map lattices \cite{liverani2004invariant,Keller2}. In  coupled map lattices, the state space is $\mathbb{T}^{\mathbb{Z}}$ and $f: \mathbb{T} \rightarrow \mathbb{T}$ a smooth uniformly expanding map describing the dynamics at each node, and $F: \mathbb{T}^{\mathbb{Z}} \rightarrow \mathbb{T}^{\mathbb{Z}}$ is a product of these maps describing the global uncoupled evolution. Next, the coupling is described by $H_{\varepsilon} : \mathbb{T}^{\mathbb{Z}} \rightarrow \mathbb{T}^{\mathbb{Z}}$ where  $DH_{\varepsilon} = 1 + \varepsilon A$, where 
$A_{ij}$ and its derivatives have elements decaying exponentially fast as a function of the distance of the matrix elements. Finally, the coupled dynamical system considered is $F_{\varepsilon} =  H_{\varepsilon} \circ F$. 
Then, the authors prove that  $F_{\varepsilon}$ has an invariant measure and its projection on every finite block of sites has absolutely continuous marginals. 
By contrast, in the graphon framework, finite subsets of nodes collapse to sets of vanishing Lebesgue measure as $N\to\infty$, so there is no nontrivial analogue of finite block marginals. Instead, the meaningful limit object is the empirical distribution of all nodes, encoded as a measure on $[0,1]\times T$ with disintegration $\{\nu_z\}_{z\in[0,1]}$. Regularity in this setting is expressed fiberwise: for almost every $z$, the conditional measure $\nu_z$ is absolutely continuous with a density of bounded variation (or even $C^2$).

\end{example}

\begin{example}[Erd\"os-Renyi Random Graphs]  An Erd\"os-Renyi random network on $N$ vertices has adjacency matrix with entries $A^N_{ij}$ that are independent Bernoulli random variables with probability $p\in[0,1]$. Consider $\W$ the graphon representation of a sequence of Erd\"os-Renyi random graphs with probability $p\in[0,1]$ and denote by $(\Omega,\mb P)$ the underlying probability space. It is expected that $\W$ should converge to the constant graphon $W$ taking constant value equal to $p$. We are going to show that,  with this limit, hypothesis (H2) becomes true. 

Pick a measure  $\nu$ with constant disintegration, i.e. $\nu_z=\mu$ with $\mu$ a measure on $\T$ for all $z\in[0,1]$. Notice that, in this case, hypothesis (H2) is satisfied if for every $z\in E\subset[0,1]$, $\int_0^1dz'\W(z,z')\rightarrow p$. We are now going to show that this limit holds with $E=[0,1]$ and almost surely in the realisation of the sequence $\W$. For every $N\in \N$ and $1\le i\le N$ define the sets
\[
\Omega_{N,i}:=\left\{\omega\in \Omega:\,\left|\frac{1}{N}\sum_{j=1}^NA^N_{ij}(\omega)-p\right|>N^{-\frac13}\right\}.
\]
Since $\{A_{ij}^{(N)}\}_{j=1}^N$ are i.i.d. Bernoulli random variables, by the Hoeffding inequality 
$
\mb P(\Omega_{N,i})\le 2\exp(-2N^{1/3})
$; therefore, defining $\Omega_N:=\cup_{i=1}^N\Omega_{i,N}$ we have $\mb P(\Omega_N)\le  2N\exp(-2N^{1/3})$. Since $\sum_{N\ge 1}\mb P(\Omega_N)<\infty$, the Borel-Cantelli lemma implies that there is $\Omega'\subset\Omega$ with $\mb P(\Omega')=1$, such that for any $\omega\in \Omega'$ there is $\bar N$ such that for any $N\ge \bar N$, $\omega\in \Omega_N^c$. On the other end, if $\omega\in \Omega_{N}^c$,  by definition of $\W$
\[
\left|\int_{0}^1dz'\W(\omega)(z,z')-p\right|<N^{-\frac13} \quad\quad\forall z\in[0,1],
\]
which implies the claim.
\end{example}

\begin{remark} For the $\nu$ and $W$ satisfying hypothesis (H2) in the examples above, and for $\mu^{(N)}$ verifying hypothesis (H1), one has that for any finite $t\in\N$, the hypotheses (H1)-(H2) are verified also for $\tilde \nu:=\mc F^t\nu$ and $\tilde \mu^{(N)}:=(F^t)_*\mu^{(N)}$: the convergence hypothesis (H2) follows from the fact that $\mc F^t\nu$ has the same regularity properties we assumed for $\nu$, and hypothesis (H1) holds as argued in Remark \ref{Rem:HypH1}. This suggests that one can apply Theorem \ref{ConvergEmpiricalMeasure} iteratively, and that, for sufficiently large $N$, $\mc F^t\nu$ approximates the marginals of $(F^t)_*\mu^{(N)}$ for any bounded $t\in\N$.
\end{remark}

\subsection{Existence and stability of the fixed point for the STO on a graphon}

{

Next, we show that under suitable assumptions, the STO has a fixed point. The main challenge is that each node may experience a distinct local mean field due to its connectivity profile (encoded in the graphon $W$), which is a rather different situation from the full permutation symmetry case studied in previous works. From a technical point of view, we need to find a class of measures $\mc A_{\bo M}$ invariant under the STO to which we can apply Schauder's fixed point theorem; we need some compactness property for this set that we obtain by requiring that $\nu\in \mc A_{\bo M}$ has disintegration $\{\nu_z\}_{z\in[0,1]}$ with some regularity in the $z$ coordinate. This will allow us to deduce compactness (in a suitable sense) of the class $\mc A_{\bo M}$, and apply a fixed point theorem.
}

{

\begin{theorem}[Existence and uniqueness of an attracting fixed point]\label{main}
Let $ f \in  \mathcal{C}^{3}(\mathbb{T},\mathbb{T})$ be a uniformly expanding map,  ${h\in \mathcal{C}^{3}(\mathbb{T}\times\mathbb{T},\mathbb{R})}$,  $W\in L^{\infty}([0,1]^2)$ be a graphon with $\mathtt{var}_{p,L^1}(W)<\infty$ for some $p\in(0,1]$, and $\mc F$ the associated STO defined in \eqref{Eq:STODef}.

\begin{itemize}
	\item[(i)] Consider the admissible set of regular densities $\mc A_{\bo M}$ as in \eqref{Eq:SpaceofMeasures}.  Then there exist $\alpha_0>0$ and $M_0, \,M_1,\,M_2>0$ such that for all $|\alpha|< \alpha_0$,  $\mc F$ has a fixed point $\phi^\star$ in the closure of $\mc A_{\bo M}$ in $\mc B_w$.
	\item[(ii)] Furthermore, the fixed point is unique and locally exponentially stable meaning that there is $\delta>0$ and constants $K>0,\,\rho>0$ such that for every $\phi\in \mc M_{1,\Leb}$ with $\|\phi_z-\phi^\star_z\|_{\mc C^1}\le \epsilon$  for a.e. $z\in[0,1]$,  the following holds
\[
\sup_{x\in \mathbb{T}}\vert (\mc F^t \phi)_{z}(x)-\phi^{*}_{z}(x)\vert \le K e^{- \rho t}. 
\]
\end{itemize}

\end{theorem}
}

For the fixed point found above, the following holds:
\begin{corollary}\label{Cor:SmoothnessFixedPoint}
Let  $\phi_*$ be fixed point of $\mc F$ with absolutely continuous disintegration, then $\varphi^{*}_{z} \in \mc C^2 (\mb T, \mb R)$ for a.e. $z\in [0,1]$.
\end{corollary}
{
This result guarantees the existence of a fixed density under a mild regularity and a small coupling condition, showing that the collective behavior of the system can settle into a statistically stable configuration. This fixed point represents a Selfconsistent equilibrium in which the distribution of states across the network aligns with the distribution induced by the dynamics. Proving the existence of such a fixed point provides a foundation for analyzing long-term behaviors, such as stability and convergence of empirical distributions which will be addressed in the following sections.

\subsection{Lipschitz Continuity of the STO}

The next result shows that the STO is Lipschitz both in the measure it acts on and on the graphon that defines it.
\begin{theorem}[Lipschitz Continuity of $\mc F$]\label{convergence_STO}
Let $ f \in  \mathcal{C}^{1}(\mathbb{T},\mathbb{T})$ be an expanding map, ${h\in \mathcal{C}^{1}(\mathbb{T}\times\mathbb{T},\mathbb{R})}$ the coupling function. Then the mapping $(W,\phi)\mapsto \mc F\phi$ is Lipschitz continuous from $L^1([0,1],L^\infty([0,1]))\times \mc B_w$ to $\mc B_w$, i.e. there exists a constant $K>0$ such that for any $W,\, \tilde W\in L^1([0,1],L^\infty([0,1]))$, calling $\mc F$ and $\tilde{\mc F}$ the STO associated to $W$ and $\tilde W$ respectively, one has
\[
\|\mc F\phi-\tilde{\mc F}\tilde \phi\|_{``1"}\le K\left[\|W-\tilde W\|_{L^1}+\|\phi-\tilde \phi\|_{``1"}\right]
\]
where $K$ depends on $\|\tilde W\|_{L^1([0,1],L^\infty([0,1]))}$.
\end{theorem}

\bigskip
{\bf Acknowledgments}. We thank Alice de Lorenci, Bella Figliaggi, and 
Zheng Bian for enlightening discussions. TP was supported by the FAPESP Grant No. 2013/07375-0, Serrapilheira Institute (Grant No. Serra-1709-16124), and CNPq grant 312287/2021-6. TP and MT thank the support of EPSRC-FAPESP Grant No. 2023/13706. MT thanks the support of EPSRC grant number UKRI1021.

\section{Proofs}

\subsection{Convergence of the finite dimensional system to the STO} 

Before proceeding with the proof of Theorem \ref{ConvergEmpiricalMeasure}, we give two lemmas that elucidate hypothesis (H2).

\begin{lemma}\label{Lem:H2Suf11} Assume that there is $E\subset [0,1]$ such that for every $z_*\in E$,  $W^{(N)}(z_*,\cdot)\rightarrow W(z_*,\cdot)$ in $L^1([0,1])$ for $N\rightarrow \infty$. Then (H2) is satisfied. 
\end{lemma}
\begin{proof}
By H\"older inequality, for any $z_*\in[0,1]$ and $x\in \T$
\begin{align*}
&\left|\int_{0}^{1} dz \, \W(z_*,z) \int_\T h(x,y)d\nu_{z}(y) - \int_{0}^{1} dz \, W(z_*,z) \int_\T h(x,y)d\nu_{z}(y)\right| \le\\
&\quad\quad\le \|h\|_{\mc C^0}\|\W(z_*,\cdot)-W(z_*,\cdot)\|_{L^1}.
\end{align*}
One can see that if $z_*\in E$, $\|\W(z_*,\cdot)-W(z_*,\cdot)\|_{L^1}\rightarrow 0$ for $N\rightarrow \infty$ and hypothesis (H2) holds.
\end{proof}

Requiring almost everywhere convergence of the graphon in $L^1$ is too restrictive in some cases of interest (e.g. for Erd\"os-R\'enyi networks). Therefore we provide the following alternative convergence criterion.

\begin{lemma}\label{Lem:H2Suf12}  Assume that there is $E\subset [0,1]$ such that for every $z_*\in E$: i) $W^{(N)}(z_*,\cdot)\rightarrow W(z_*,\cdot)$ weakly for $N\rightarrow \infty$\footnote{As for measures, we mean that \[
\lim_{N\rightarrow\infty } \sup_{
\substack{\Lip(g)\leq 1 \\ \Vert g \Vert_{\infty}\leq 1 } } \int_0^1dz' g(z')[W^{(N)}(z_*,z')-W(z_*,z')]=0
\]}, and ii)   $\nu \in \mc M_{1,\Leb}([0,1]\times\T)$ has a disintegration $\{\nu_z\}_{z\in[0,1]}$ that is Lipschitz on the whole of $[0,1]$. Then (H2) is satisfied.
\end{lemma}
\begin{proof}
If $z\mapsto\nu_z$ is $K$-Lipschitz with respect to the Wasserstein metric, then, since $h$ is $\mc C^1$, this implies that for any $x\in \T$, $z\mapsto \psi(z):=\int_\T h(x,y)d\nu_z(y)$ is $K\|h\|_{\mc C^1}$-Lipschitz; in fact,
\[
\left| \int_\T h(x,y)d\nu_z(y)- \int_\T h(x,y)d\nu_{z'}(y)\right|\le \|h\|_{\mc C^1}\|\nu_z-\nu_{z'}\|_{W^1}\le \|h\|_{\mc C^1}K|z-z'|.
\]
This implies that 
\begin{align*}
&\left|\int_{0}^{1} dz \, \W(z_*,z) \int_\T h(x,y)d\nu_{z}(y) - \int_{0}^{1} dz \, W(z_*,z) \int_\T h(x,y)d\nu_{z}(y)\right| \le\\
&\quad\quad  \le \left|\int_{0}^{1} dz \, (\W(z_*,z)-W(z_*,z))\psi(z)\right| \\
&\quad\quad  \le \|h\|_{\mc C^1}K\|\W(z_*,\cdot)-W(z_*,\cdot)\|_{W^1},
\end{align*}
and it is easy to see from the above that hypothesis (H2) holds.
\end{proof}

\begin{proof}[Proof of Theorem \ref{ConvergEmpiricalMeasure}] 
Let's fix $z_*\in E$. To shorten notation, let's call $i=i(N):= \lfloor z_*N\rfloor$.

 Take any function $g$ Lipschitz on $\T$ with Lipschitz constant equal to one. From the definitions of $\Pi_iF_*\mu^{(N)}$ and $\mc F\nu$ we obtain:
\begin{align}
\int_\T g(y)\left(d (\mc F\nu)_{z_*}(y)-d\Pi_iF_*\mu^{(N)}(y)\right)&=\int_\T \, g(F_{\nu,z_*}(y))d\nu_{z_*}(y)-\label{Eq:Int1}\\
 &\quad\quad-\int_{\T^N} g\left(f(x_i)+\frac{\alpha}{N}\sum_{j=1}^NA_{ij}h(x_i,x_j)\right)d\mu^{(N)}(x_1,...,x_N)\label{Eq:Int2}\\
 &=\mathcal I_1-\mathcal I_2.\nonumber
\end{align}
To give estimates for the integral $\mathcal I_2$ in \eqref{Eq:Int2}, we are going to use the fact that $\mu^{(N)}$ satisfies a concentration inequality and the fact that the average function $\psi(x_1,...,x_N):=\frac{1}{N}\sum_{j=1}^NA_{ij}h(x,x_j)$, for any fixed $x\in \T$, is Lipschitz of order $\|h\|_{\mc C^1}N^{-1}$ in its entries and therefore close to its expectation w.r.t. $\mu^{(N)}$; informally,
\begin{align*}
g\left(f(x_i)+\frac{\alpha}{N}\sum_{j=1}^NA_{ij}h(x_i,x_j)\right) &\approx g\left(f(x_i)+\int\frac{\alpha}{N}\sum_{j=1}^NA_{ij}h(x_i,x_j')d\mu^{(N)}(x_1',...,x_N')\right)\\
&= g(F_{\nu,z_*}(x_i))+O(N^{-1/3})
\end{align*}
uniformly in $N$ and this will let us conclude that the difference between $\mc I_1$ and $\mc I_2$ is $O(N^{-\frac13})$; rigorously proving this requires some steps.

\emph{Step 1}:  First notice that by the concentration inequality \eqref{Eq:H1} satisfied by $\mu^{(N)}$, for every $x\in \T$, one has
\begin{align*}
\mu^{(N)}\left(\left|\frac{1}{N}\sum_{j=1}^NA_{ij}h(x,x_j)-\int_{\T^N} \frac1N\sum_{j=1}^NA_{ij}h(x,x_j')d\mu^{(N)}(x_j)\right|>{N^{-1/3}} \right)&\le C_1\exp\left[-C_2N^{1/3}\right].
\end{align*}

\emph{Step 2}:  Fix $y_1,...,y_{M}$ an $\epsilon>0$ net, meaning that for every $x\in \T$ there is $k$ such that $|x-y_k|<\epsilon$ \footnote{ Here $M$ depends on $\epsilon$.}. Then define
\[
\Omega_k:=\left\{(x_1,...,x_N)\in \T^N:\, \left|\frac{1}{N}\sum_{j=1}^NA_{ij}h(y_k,x_j)-\mb E_{\mu^{(N)}}\left[  \frac1N\sum_{j=1}^NA_{ij}h(y_k,x_j)\right]\right|>N^{-1/3}\right\}.
\]
Let $\Omega:=\cup_{k=1}^M\Omega_k$, notice that $\mu^{(N)}(\Omega)\le MC_1\exp\left[-C_2N^{1/3}\right]$. Then, choosing a net with $\epsilon=N^{-1/3}$, for any $x\in \T$ and $(x_1,...,x_N)\in \Omega^c$, an application of triangle inequality gives:
\[
 \left|\frac{1}{N}\sum_{j=1}^NA_{ij}h(x,x_j)-\mb E_{\mu^{(N)}}\left[  \frac1N\sum_{j=1}^NA_{ij}h(y_k,x_j)\right]\right|<O(N^{-1/3}).
 \]

\emph{Step 3}: $\mathcal I_2$ can be approximated as
\begin{align*}
&\int_\T d\mu_i^{(N)}(x_i)\int_{\T^{N}} g\left(f(x_i)+\frac{\alpha}{N}\sum_{j=1}^NA_{ij}h(x_i,x_j')\right) \prod_{j=1}^Nd\mu^{(N)}(x_1',...,x_N')+O(N^{-1}) =\\
&\quad\quad =\int_\T d\mu_i^{(N)}(x_i)\left(\int_{\Omega} +\int_{\Omega^c}\right)g\left(f(x_i)+\frac{\alpha}{N}\sum_{j=1}^NA_{ij}h(x_i,x_j')\right)d\mu^{(N)}(x_1',...,x_N')+O(N^{-1}) \\
&\quad\quad\le C_\# N^{1/3}\exp\left[-cN^{1/3}\right]+\int_\T d\mu_i^{(N)}(x_i)\int_{\Omega^c}g\left(f(x_i)+\frac{\alpha}{N}\sum_{j=1}^NA_{ij}h(x_i,x_j')\right) d\mu^{(N)}(x_1',...,x_N')+O(N^{-1})\\
&\quad\quad\le \int_\T d\mu_i^{(N)}(x_i)\int_{\Omega^c}g\left(f(x_i)+\frac{\alpha}{N}\sum_{j=1}^NA_{ij}h(x_i,x_j')\right)\mu^{(N)}(x_1',...,x_N')+O(N^{-1})\\
&\quad\quad\le \int_\T d\mu_i^{(N)}(x_i) \mathbb P(\Omega^c) g\left(f(x_i)+\mb E_{\mu^{(N)}}\left[\frac{\alpha}{N}\sum_{j=1}^NA_{ij}h(x_i,x_j')\right]\right)+ O(N^{-1/3})\\
&\quad\quad\le  \int_\T d\mu_i^{(N)}(x_i)   g\left(f(x_i)+\frac{\alpha}{N}\sum_{j=1}^N\int_{\T}A_{ij}h(x_i,x_j')d\mu_j^{(N)}(x_j')\right)+ O(N^{-1/3})
\end{align*}
where we used the estimate  $\mathbb P(\Omega^c)\ge 1-C_1M\exp[-C_2N^{1/3}]$. Recalling the definition of $\mu_j^{(N)}$ and of  $\W$ we have that 
\begin{align*}
\frac{1}{N}\sum_{j=1}^N\int_\T A_{ij}h(x_i,x_j')d\mu_j^{(N)}(x_j') &= \frac{1}{N}\sum_{j=1}^N\int_\T A_{ij}h(x_i,x_j')N\int_{j/N}^{(j+1)/N}dz' d\nu_{z'}(x_j')\\
&= \sum_{j=1}^N\int_\T \int_{j/N}^{(j+1)/N}dz' \W(z,z')h(x_i,x_j') d\nu_{z'}(x_j')
\end{align*}
for all $z\in((i-1)/N,i/N]$, and in particular for $z_*$, 
\begin{equation}\label{Eq:UpToEqual}
f(x_i)+\frac{\alpha}{N}\sum_{j=1}^N\int_{\T}A_{ij}h(x_i,x_j')d\mu_j^{(N)}(x_j')=f(x_i)+\alpha\int_\T\int_{0}^{1} dz' \, \W(z_*,z') h(x_i,x_j')d\nu_{z'}(x_j').
\end{equation}
Now using \eqref{Eq:CH}, for any $\delta>0$, there is $N_\delta$ such that for $N\ge N_\delta$ and all $x_i\in \T$,
\begin{align*}
f(x_i)+\alpha\int_\T\int_{0}^{1} dz' \, \W(z_*,z') h(x_i,x_j')d\nu_{z'}(x_j')&=f(x_i)+\alpha\int_\T\int_{0}^{1} dz' \, W(z_*,z') h(x_i,x_j')d\nu_{z'}(x_j')+\delta\\
&=F_{\nu,z_*}(x_i)+\delta.
\end{align*}

Putting together all the estimates obtained so far:
\begin{align*}
\int_{\T^N} g\left(f(x_i)+\frac{\alpha}{N}\sum_{j=1}^NA_{ij}h(x_i,x_j)\right)d\mu^{(N)}(x_1,...,x_N) &= \int_\T d\mu_i^{(N)}(x_i)g(F_{\nu,z_*}(x_i))+O(\delta)+O(N^{-1/3})\\
&=\int_\T d\nu_{z_*}(x_i)g(F_{\nu,z_*}(x_i))+O(\delta)+O(N^{-1/3})
\end{align*}
where in the last line we used that $d_W(\mu_i^{(N)},\nu_{z_*})=O(N^{-1})$ and $g\circ F_{\nu,i/N}$ is Lipschitz. This allows us to conclude that
\[
\mathcal I_1-\mathcal I_2=O(\delta)+O(N^{-1/3})
\]
and the claim follows because $\delta$ is arbitrary.
\end{proof}

\subsection{Proof of Existence of the fixed point: Theorem \ref{main}}

The proof of the theorem can be obtained following the steps below:
\begin{description}
\item[Existence of an invariant set of densities for an iterate of $\mc F$.]  $\mc A_{\bold M}$ under the action the STO.  Crucially, one can show that for sufficiently low coupling the strength, an iterate of the STO leaves invariant a set of measures in $\mc A_{\bo M}$ that have disintegrations $\{\nu_z\}_{z\in[0,1]}$ with bounded oscillation with respect to $z$. 
We obtain this result as a consequence of three auxiliary results. In Proposition \ref{kthDerivative}, we control the dependence of the maps $F_{\phi,z}$  on  $z\in[0,1]$; in Proposition \ref{Lem:UnifProponFiber}, we show that there are sets of measures invariant under the STO for which the regularity of $\frac{d\nu_z}{d\Leb_{\T}}$ is controlled; in Proposition \ref{DifDynSameDensityDERIVADA}, we bound the dependence of the action of the STO on different fibers in terms of the graphon. Finally, we prove in Lemma \ref{F-invariante} that there is an invariant set of densities where the action of the STO has controlled distortion.

\item[Convexity and pre-compactness and of the invariant set.]
The control in $z$ of the disintegration of the measures in $\mc A_{\bold M}$ allows us to conclude in Lemma \ref{RelativamenteCompacto} that $\mc A_{\bold M}$ is relatively compact in $\mathcal{B}_{w}$. We  show that $\mc A_{\bold M}$ is convex in Lemma \ref{espacioConvexo}.

\item[The STO is Lipschitz in the weak norm.] We prove that the STO restricted to the strong space $\mathcal{B}_{s,M}\cap \mc M_{1,\Leb}\supset \mc A_{\bo M}$ is Lipschitz continuous in the weak norm $\Vert \cdot \Vert_{``1"}$  (Lemma \ref{cont-STO}). 

\item[Existence of a fixed density for an iterate of $\mc F$.]  Combining the above steps, we apply Schauder's fixed point theorem to deduce the existence of a fixed point for a sufficiently large iterate of $\mc F$ (i.e. a periodic orbit for $\mc F$).

\item[Existence, stability, and uniqueness of a fixed point for $\mc F$] We introduce a projective distance with respect to which $\mc F$ is a contraction. We use this to show that the fixed point for an iterate of $\mc F$ is actually a fixed point for $\mc F$, and that it is locally attracting.

\end{description}

\subsubsection{Existence of an invariant set of densities $\mc A_{\bo M}$ for an iterate of $\mc F$}

The coupling strength $\alpha$ plays an important role in our results, for instance, when the coupling is small, the property that the local dynamics $f$ is uniformly expanding carries over to the map $F_{\varphi,z}$. This can be observed  in the following:

\begin{lemma}[Uniform Expansion and Distortion Bounds for Fiber Maps]\label{F-Expanding}
Assume  $f$ and  $h$ are $\mc C^3$, and $W\in L^{\infty}([0,1],L^1([0,1]))$. Consider $\nu\in \mc M_{1,\Leb}$ and $F_{\nu,z}$ the restriction of the dynamics to the fiber $\{z\}\times\mathbb{T}$ as defined in \eqref{Eq:SetDynFiber}.
If 
\begin{equation}\label{Eq:ConUnExp}
|\alpha|\cdot \| W\|_{L^{\infty}([0,1],L^1([0,1],\R))}<\hat{\alpha}:= \frac{\inf_{x\in \T} |f^{\prime}(x)|  - 1}{\Vert h\Vert_{\mc C^1}},
\end{equation}
 then for any $\nu\in \mc M_{1,\Leb}$ and a.e. $z\in[0,1]$, the map $F_{\nu, z}$ is  uniformly expanding. Moreover, there are $K,K'>0$ such that for any $\nu\in \mc M_{1,\Leb}$ and for a.e. $z\in[0,1]$ 
\[
\|F_{\nu,z}\|_{\mc C^3}<K, \mbox{~and~} 
\sup_{x\in\T}\left|\frac{F_{\nu,z}''(x)}{(F_{\nu,z}'(x))^2}\right|<K'.
\]

\end{lemma}
\begin{proof}
Note that for every $z\in[0,1]$ for which $W(z,\cdot)\in L^1([0,1],\R)$, the $i$-th derivative of $F_{\nu,z}$ is given by,
\begin{equation*}
F_{\nu, z}^{(i)}(x) = f^{(i)}(x) + \alpha \int_{0}^{1} dz' \int_{\mathbb{T}} \, \,  \, W(z, z') \, \partial_{1}^{(i)} h(x, y) \, d\nu_{z'}(y)
\end{equation*}
leading to, for a.e. $z\in[0,1]$,
\begin{align*}
\vert F'_{\nu, z}(x) \vert & \geq  \vert f'(x) \vert - |\alpha|\left| \int_{0}^{1} dz' \left( \int_{\mathbb{T}} \, \partial_{1} h(x, y) \, d\nu_{z'}(y) \right) \, \, W(z, z') \right|\\
	& \geq \vert f'(x) \vert - |\alpha| \Vert \partial_{1} h(x, \cdot)\Vert_{\infty} \, \Vert W(z, \cdot)\Vert_{L^1}\\
	&\geq \inf_{x\in\mathbb{T}} \vert f'(x)\vert - \alpha \, \|h\|_{\mc C^1}  \| W\|_{L^{\infty}([0,1],L^1([0,1],\R))}
\end{align*}
where we used that $\nu_{z'}$ is a probability measure and H\"older inequality; and, analogously, 
\[
|F_{\nu, z}^{(i)}(x)| \le \sup_{x\in \mb T}|f^{(i)}(x)|+|\alpha|\Vert h\Vert_{\mc C^{i}}  \| W\|_{L^{\infty}([0,1],L^1([0,1],\R))}:=B_{i}.
\]

To ensure that $\min_{x\in\mathbb{T}}\vert F'_{\nu, z}(x)\vert > 1 $, we require
\begin{equation*}
 \inf_{x\in\mathbb{T}}\vert f'(x)\vert - \alpha \, \Vert h\Vert_{\mc C^1}  \| W\|_{L^{\infty}([0,1],L^1([0,1],\R))}> 1
\end{equation*}
and  since $\vert f'(x)\vert \geq \sigma > 1$ the first claim follows.

To estimate the $\mc C^3$ norm of the local dynamics, note that for a.e. $z\in [0,1]$ and $\nu\in \mc M_{1,\Leb}$ we have
\begin{align*}
\Vert F_{\nu, z}\Vert_{\mc C^3} & := \sup_{x\in \mb T}|F_{\nu, z}(x)|+\sup_{x\in \mb T}|F'_{\nu, z}(x)|+\sup_{x\in \mb T}|F''_{\nu, z}(x)|+\sup_{x\in \mb T}|F'''_{\nu, z}(x)|\\
& \le \Vert f \Vert_{\mc{C}^3}+ |\alpha|\Vert h \Vert_{\mc{C}^3}\| W\|_{L^{\infty}([0,1],L^1([0,1],\R))}:=K.
\end{align*}
so, the fiber maps $F_{\nu,z}$ belong to $\mc C^3(\mathbb{T})$ for any $\nu \in \mathcal{M}_{1,\mathrm{Leb}}$ and for a.e. $z \in [0,1]$. To estimate the distortion:
\[
\sup_{x\in \mb T}\left| \frac{F''_{\nu, z}(x)}{(F'_{\nu, z}(x))^{2}} \right| \le \frac{\sup_{x\in \mb T}|F''_{\nu, z}(x)|}{\inf_{x\in \mb T}|(F'_{\nu, z}(x))^{2}|} :=K'<\infty
\]
where $K'$ is independent of $\nu$ and $z$.
\end{proof}

\begin{definition}\label{Def:xi}
From now on, for $|\alpha|$ under the assumptions of the above lemma, we will denote by $\xi=\xi(\alpha)>1$ the minimal expansion rate of $F_{\nu,z}$ that is, $\vert F'_{\nu,z}(x)\vert \geq \xi(\alpha)>1$ for every $\nu\in \mc M_{1,\Leb}$ and a.e. $z\in[0,1]$.
\end{definition}

The above lemma can be used to show that there are invariant sets of measures under the STO for which the conditional measures on fibers $\{z\}\times\T$ have a smooth density: Recall the definition of $\tilde {\mc B}_{BV^i,M_i}$ given in \eqref{Set:RegOnFib}, the following holds
\begin{proposition}\label{Lem:UnifProponFiber}
Under the assumptions of Lemma \ref{F-Expanding} and condition \eqref{Eq:ConUnExp}, there are $M_1,\,M_2>0$  such that $\mc F(\tilde {\mc B}_{BV^1,M_1}\cap \tilde{\mc B}_{BV^2,M_2}\cap \mc M_{1,\Leb})\subset\tilde {\mc B}_{BV^1,M_1}\cap \tilde{\mc B}_{BV^2,M_2}\cap \mc M_{1,\Leb}$. 
\end{proposition}
\begin{proof}
 It has  been established that $\mc F(\mc M_{1,\Leb})\subset \mc M_{1,\Leb}$. 
To conclude the proof we need to show that for sufficiently large $M_1,\,M_2>0$, $\nu_z\in \mc B_{BV^1,M_1}\cap\mc B_{BV^2,M_2}$ for a.e. $z$ implies that $(F_{\nu,z})_*\nu_z\in\mc B_{BV^1,M_1}\cap\mc B_{BV^2,M_2}$ for a.e. $z$. From Lemma \ref{F-Expanding}  the maps $F_{\nu,z}$ have a common lower bound (uniform in $\nu$ and for a.e. $z\in[0,1]$) on their minimal expansion and a common upper bound on the $\mc C^3(\mathbb{T})$ norm and their distortion. Applying Lemma \ref{Lem:BoundUnifExpMaps}, the result follows.
\end{proof}

We now study the regularity of $F_{\nu,z}$ with $z \in [0,1]$. Assume that $W \in L^{\infty}([0,1], L^1([0,1],\R))$. By definition, this means that $W: [0,1]^2 \to \mathbb{R}$ is measurable and that for almost every $z \in [0,1]$, the section $W(z, \cdot)$ belongs to $L^1([0,1])$, with essentially bounded $L^1$-norm. 
\begin{definition}Call
\[
A_{W} := \left\{ z \in [0,1] : W(z,\cdot) \in L^1([0,1]) \right\}
\]
the set of full measure for which the section $W(z,\cdot)$ is well-defined in $ L^1([0,1])$.
\end{definition}

\begin{proposition}\label{kthDerivative}
For $k\ge 0$, let $f\in\mathcal{C}^{k}(\mathbb{T},\mathbb{T})$, $h\in\mathcal{C}^{k}(\mathbb{T}\times\mathbb{T},\mathbb{R})$, $\nu\in \mc M_{1,\Leb}$ and $W\in L^{\infty}([0,1],L^1([0,1],\R))$ then, for $z,\bar{z} \in A_{W}$ we have that
    $$d_{\mc C^k} (F_{\nu,z},F_{\nu,\bar{z}}):=\sum_{i=0}^{k}\sup_{x\in\T}   \left\vert F^{(i)}_{\nu,z}(x)-F^{(i)}_{\nu,\bar{z}}(x) \right\vert   \leq \alpha \, \Vert h\Vert_{\mc C^k} \Vert W(z,\cdot)-W(\bar{z},\cdot)\Vert_{L^1}.$$    
\end{proposition}
\begin{proof}
   Consider the $k$-th derivative of the dynamics in the fiber $\{z\}\times\mathbb{T}$, we estimate
\begin{align*}
   \left\vert F^{(i)}_{\nu,z}(x)-F^{(i)}_{\nu,\bar{z}}(x) \right\vert & = \left\vert \alpha \int_{0}^{1}dz'\int_{\mathbb{T}} \, (W(z,z')-W(\bar{z},z')) \, \partial_{1}^{(i)} h(x,y) \, d\nu_{z'}(y) \right\vert  \nonumber \\
         & \leq \alpha \, \Vert W(z,\cdot)-W(\bar{z},\cdot)\Vert_{L^1} \sup_{z'\in [0,1]} \left|\int_{\mathbb{T}} \partial_{1}^{(i)} h(x,y) d\nu_{z'}(y)\right|\nonumber \\
         & \leq \alpha \, \sup_{x,y\in\T}|\partial_1^{(i)}h(x,y)|\Vert W(z,\cdot)-W(\bar{z},\cdot)\Vert_{L^1}
\end{align*}
and the result follows.
\end{proof}

\begin{proposition}\label{DifDynSameDensityDERIVADA} Under the assumption of Lemma \ref{F-Expanding} and condition \eqref{Eq:ConUnExp}, for any $M_1,\,M_2>0$ there is a constant $K_{\#}:=K_{\#}(\alpha)$ with $K_\#>0$ such that for any $\phi\in \mc M_{1,\Leb}\cap \tilde {\mc B}_{BV^1,M_1}\cap \tilde {\mc B}_{BV^2,M_2}$, 
\begin{equation*}
|(F_{\varphi,z})_{*}\varphi_{\bar{z}} - (F_{\varphi,\bar{z}})_{*}\varphi_{\bar{z}} |_{BV^1} \leq K_\# \Vert W(z,\cdot)-W(\bar{z},\cdot)\Vert_{L^1}
\end{equation*}
for all $\bar z,z\in A_{W} \subset [0,1]$ and $\bar z$ in the full measure set for which $\phi_{\bar z}\in \mc B_{BV^1,M_1}\cap\mc B_{BV^2,M_2}$.
\end{proposition}

\begin{proof}
Fix $\phi\in \mc M_{1,\Leb}\cap \tilde {\mc B}_{BV^1,M_1}\cap \tilde {\mc B}_{BV^2,M_2}$ and pick any $\psi\in \mc B_{BV^2}$ and $g\in \mc C^1(\T,\R)$ with $\Vert g \Vert_\infty\le 1$. To ease notation, we will drop the $\phi$ dependence from $F_{\phi,z}$ that will be denoted by $F_z$.
\begin{align}
|(F_{z})_*\psi-(F_{\bar z})_*\psi|_{BV^1}&\le \int_\T  g'[(F_{z})_*\psi-(F_{\bar z})_*\psi] \nonumber\\
&= \int_\T (g'\circ F_{z}-g'\circ F_{\bar z})\psi \nonumber\\
&=\int_\T\left[\left(\frac{g\circ F_{z}}{F_{z}'}\right)'-\left(\frac{g\circ F_{\bar z}}{F_{\bar z}'}\right)'\right]\psi+\int_\T\left[\frac{g\circ F_{z}}{(F_{z}')^2}F_z''-\frac{g\circ F_{\bar z}}{(F_{\bar z}')^2}F_{\bar z}''\right]\psi \nonumber\\
&=\int_\T\left(\frac{g\circ F_{z}}{F_{z}'}-\frac{g\circ F_{\bar z}}{F_{\bar z}'}\right)'\psi+\int_\T\left[\frac{g\circ F_{z}}{(F_{z}')^2}F_z''-\frac{g\circ F_{\bar z}}{(F_{\bar z}')^2}F_{\bar z}''\right]\psi.\label{Eq:LipFibBV1}
\end{align}
We start by treating the first integral. Modifying some arguments that can be found in \cite{liverani2019statisticalpropertiesuniformlyhyperbolic}, we can consider $\Phi\in \mc C^2(\mb T,\R)$ defined by
\[
\Phi(x):=(F_{z}')^{-2}(x)\int_{F_{\bar z}(x)}^{F_{z}(x)}ds\,g(s)
\]
with derivative
\[
\Phi'(x)=-2(F_{z}')^{-1}(x)F_{z}''(x)\Phi(x)-F_{\bar z}'(x)(F_{z}'(x))^{-2}g(F_{\bar z}(x))+\frac{g(F_{z}(x))}{F_{z}'(x)},
\]
and notice that
\begin{align*}
\int_\T\left(\frac{g\circ F_{z}}{F_{z}'}-\frac{g\circ F_{\bar z}}{F_{\bar z}'}\right)'\psi&= \int_\T \left[\Phi' + 2(F_{z}')^{-1}F_{z}''\Phi+F_{\bar z}'(F_{z}')^{-2}g\circ F_{\bar z}-\frac{g\circ F_{\bar z}}{F_{\bar z}'}\right]' \psi\\
& =\int_\T \,\Phi'' \psi+\int_\T \left[2(F_{z}')^{-1}F_{z}''\Phi+\left(-1+(F_{\bar z}')^2(F_{z}')^{-2}\right)\frac{g\circ F_{\bar z}}{F'_{\bar z}}\right]' \psi. 
\end{align*}
We now proceed to bound the two integrals above. For the first one, we need to estimate $|\Phi(x)|$.  Using that $\Vert g\Vert _\infty\le 1$, and recalling the definition of $\xi$ from Definition \ref{Def:xi}
\begin{equation*}
    |\Phi(x)|=\left| (F_{z}')^{-2}(x)\int_{F_{\bar z}(x)}^{F_{z}(x)}ds \, g(s) \right| \leq \dfrac{1}{|(F_{z}')^{2}(x)|}\int_{F_{\bar z}(x)}^{F_{z}(x)}ds \, |g(s)| \leq \dfrac{1}{\xi^2}|F_{z}(x)-F_{\bar z}(x)|
\end{equation*}
then, by Proposition \ref{kthDerivative} we have that
\[ 
|\Phi(x)| \leq \alpha\xi^{-2}\Vert h\Vert_{\mc C^0}\Vert W(\bar z, \cdot)-W(z,\cdot)\Vert_{L^1}:=\beta \Vert W(\bar z, \cdot)-W(z,\cdot)\Vert_{L^1}
\]
and this implies, by definition of $|\cdot|_{BV^2}$ norm,
\[
\int_\T \,\Phi'' \psi\le |\psi|_{BV^2}\beta \Vert W(\bar z, \cdot)-W(z,\cdot)\Vert_{L^1}.
\]
For the second one, call $G(x)$ the expression in square bracket:
\[
G(x):=2(F_{z}')^{-1}(x)F_{z}''(x)\Phi(x)+\left(-1+(F_{\bar z}')^2(x)(F_{z}')^{-2}(x)\right)\frac{g\circ F_{\bar z}}{F'_{\bar z}}(x).
\]
To estimate $|G(x)|$ note that
\[
|-1+(F_{\bar z}')^2(F_{z}')^{-2}| \leq |(F_{z}')^{-2}||(F_{ z}')^2-(F_{\bar z}')^2|\leq |(F_{z}')^{-2}| |F_{ z}'-F_{\bar z}'||F_{ z}'+F_{\bar z}'|  
\]
and, once again by Lemma \ref{kthDerivative}, estimate the above by
\[ 
|-1+(F_{\bar z}')^2(F_{z}')^{-2}| \leq 2B_1\alpha\xi^{-2}  \Vert h\Vert_{\mc C^1} \Vert W(\bar z, \cdot)-W(z,\cdot)\Vert_{L^1} = 2B_{1}\beta \Vert W(\bar z, \cdot)-W(z,\cdot)\Vert_{L^1}.
\]
Therefore, from this estimate and the estimate on $|\Phi(x)|$ obtained above,
\begin{align*}
|G(x)| &\le 2\xi^{-1}B_{2}\beta\Vert W(\bar z, \cdot)-W(z,\cdot)\Vert_{L^1}+2\xi^{-1}B_{1}\beta\Vert W(\bar z, \cdot)-W(z,\cdot)\Vert_{L^1}\\
&\le 2\xi^{-1}\beta(B_2+B_1)\|W(\bar z,\cdot)-W(z,\cdot)\|_{L^1}\\
& = \tilde{\beta}\|W(\bar z,\cdot)-W(z,\cdot)\|_{L^1}
\end{align*}
which leads to, by definition of $|\cdot|_{BV^1}$, 
\[
\int_\T G'\psi\le |\psi|_{BV^1}  \tilde{\beta}\|W(\bar z,\cdot)-W(z,\cdot)\|_{L^1}
\]
and recalling that $\psi\in  \mc B_{BV^2,M_2}$, we conclude that there is $K_\#>0$ -- depending on $f$, $h$, $\alpha$, $|\psi|_{BV^1}$, and $|\psi|_{BV^2}$ -- such that
\[
\int_\T\left(\frac{g\circ F_{z}}{F_{z}'}-\frac{g\circ F_{\bar z}}{F_{\bar z}'}\right)'\psi\le K_\# \|W(\bar z,\cdot)-W(z,\cdot)\|_{L^1}.
\]

Analogously, to treat the second integral in \eqref{Eq:LipFibBV1}, consider $\Psi \in \mc C^1(\mb T, \mb R)$ defined by
\[
\Psi(x):=F_z''(x)(F_{z}')^{-3}(x)\int_{F_{\bar z}(x)}^{F_{z}(x)}ds\,g(s)
\]
and compute
\begin{align*}
\Psi'(x)& =-3(F_{z}'(x))^{-1}F_{z}''(x)\Psi(x)+F'''_z(x)(F_{z}'(x))^{-3}\int_{F_{\bar z}(x)}^{F_{z}(x)}ds\,g(s)-F_z''(x)(F_{z}'(x))^{-3}g(F_{\bar z}(x))F'_{\bar z}(x)+\\
& \quad +F_z''(x)\frac{g(F_{z}(x))}{(F_{z}'(x))^2}.
\end{align*}
From these expressions, we get
\begin{align}
&\int_\T\left[\frac{g\circ F_{z}}{(F_{z}')^2}F_z''-\frac{g\circ F_{\bar z}}{(F_{\bar z}')^2}F_{\bar z}''\right]\psi= \nonumber\\
&\quad = \int_\T \left[ \Psi'+3(F_{z}')^{-1}F_{z}''\Psi-F'''_z(F_{z}')^{-3}\int_{F_{\bar z}}^{F_{z}}g +F_z''(F_{z}')^{-3}g\circ F_{\bar z}F'_{\bar z}-\frac{g\circ F_{\bar z}}{(F_{\bar z}')^2}F_{\bar z}'' \right]\psi \nonumber\\
&\quad=\int_\T\Psi'\psi+\int_\T\left[3(F_{z}')^{-1}F_{z}''\Psi-F_z''' (F_{z}')^{-3}\int_{F_{\bar z}}^{F_{z}}g +F_z''F_{\bar z}'(F_{z}')^{-3}g\circ F_{\bar z}- \right.\nonumber\\
&\quad\quad- F_{\bar z}''F_{\bar z}'(F_{z}')^{-3}g\circ F_{\bar z}+ \left. F_{\bar z}''F_{\bar z}'(F_{z}')^{-3}g\circ F_{\bar z}-\frac{g\circ F_{\bar z}}{(F'_{\bar z})^2}F_{\bar z}''\right]\psi \nonumber\\
&\quad=\int_\T\Psi'\psi+\int_\T\left[3(F_{z}')^{-1}F_{z}''\Psi-F_z''' (F_{z}')^{-3}\int_{F_{\bar z}}^{F_{z}}g +[F_z''-F_{\bar z}'']F_{\bar z}'(F_{z}')^{-3}g\circ F_{\bar z}+\right .\nonumber\\
&\quad\quad+\left.\left(-1+(F_{\bar z}')^3(F_{z}')^{-3}\right)\frac{g\circ F_{\bar z}}{(F'_{\bar z})^2}F_{\bar z}''\right]\psi\label{Eq:EstSecInt}
\end{align}

First of all, observe that  by the bounds in Lemma \ref{F-Expanding} and Proposition \ref{kthDerivative}, we have
\begin{align*}
    |\Psi(x)| & \le \left| \frac{F''_{z}(x)}{(F'_{z})^2(x)} \right| \left| \frac{1}{F'_{z}(x)} \right||F_{z}(x)-F_{\bar z}(x)|\\
    & \le D\xi^{-1}\alpha\Vert h\Vert_{\mc C^1} \Vert W( z, \cdot)-W(\bar z,\cdot)\Vert_{L^1}\\
    & =\beta_{1}\Vert W( z, \cdot)-W(\bar z,\cdot)\Vert_{L^1}
\end{align*}
so
\[
\int_\T\Psi'\psi\le \beta_{1}\Vert W( z, \cdot)-W(\bar z,\cdot)\Vert_{L^1} |\psi|_{BV^1}.
\]
To bound the second integral in \eqref{Eq:EstSecInt} call $H(x)$ the expression in square bracket:
\[
H(x):=3(F_{z}')^{-1}F_{z}''\Psi-(F_{z}')^{-3}F'''_z\int_{F_{\bar z}}^{F_{z}}g+[F_z''-F_{\bar z}'']F_{\bar z}'(F_{z}')^{-3}g\circ F_{\bar z}+\left(-1+(F_{\bar z}')^3(F_{z}')^{-3}\right)\frac{g\circ F_{\bar z}}{(F'_{\bar z})^2}F_{\bar z}''.
\]
Notice that
\begin{align*}
    \left|-1+(F_{\bar z}')^3(F_{z}')^{-3}\right| & \le |(F_{z}')^{-3}||(F_{\bar z}')^3-(F_{z}')^3|\\
    & \le |(F_{z}')^{-3}||F_{\bar z}'-F_{z}'||(F_{\bar z}')^2+F_{z}'F_{\bar z}'+(F_{z}')^2|\\
    & \le 3B_{1}^{2}\xi^{-3} \alpha\Vert h\Vert_{\mc C^1} \Vert W( z, \cdot)-W(\bar z,\cdot)\Vert_{L^1}.
\end{align*}
then
\begin{align*}
|H(x)|& \le 3|(F_{z}')^{-1}||F_{z}''||\Psi|+|F_z'''|| (F_{z}')^{-3}|\left| \int_{F_{\bar z}}^{F_{z}}g \right| +|F_z''-F_{\bar z}''||F_{\bar z}'||(F_{z}')^{-3}|g\circ F_{\bar z}|+\\
& \quad + \left|-1+(F_{\bar z}')^3(F_{z}')^{-3}\right|\left|\frac{g\circ F_{\bar z}}{(F'_{\bar z})^2}F_{\bar z}''\right|\\
& \le 3\xi^{-2}B_{2} D\alpha \Vert h\Vert_{\mc C^1} \Vert W( z, \cdot)-W(\bar z,\cdot)\Vert_{L^1}+B_{3}\xi^{-3} \alpha \Vert h\Vert_{\mc C^1} \Vert W( z, \cdot)-W(\bar z,\cdot)\Vert_{L^1}+\\
& \quad +B_{1}\xi^{-3}\alpha \Vert h\Vert_{\mc C^2} \Vert W( z, \cdot)-W(\bar z,\cdot)\Vert_{L^1}+3DB_{1}^{2}\xi^{-3}\alpha\Vert h\Vert_{\mc C^1}\Vert W( z, \cdot)-W(\bar z,\cdot)\Vert_{L^1}\\
& \le (3 B_{2}D\xi+B_{3}+B_{1}+3DB_{1}^{2})\xi^{-3}\alpha\Vert h\Vert_{\mc C^2}\Vert W( z, \cdot)-W(\bar z,\cdot)\Vert_{L^1}\\
& = K_{2}\Vert W( z, \cdot)-W(\bar z,\cdot)\Vert_{L^1}.
\end{align*}
So, 
\[
\int_\T H\psi\le  K_{2}\Vert W( z, \cdot)-W(\bar z,\cdot)\Vert_{L^1}\|\psi\|_{L^1}.
\]
Combining all the estimates we obtain 
\[
\int_\T\left[\frac{g\circ F_{z}}{(F_{z}')^2}F_z''-\frac{g\circ F_{\bar z}}{(F_{\bar z}')^2}F_{\bar z}''\right]\psi  \le  K_\# \|W(z,\cdot)-W(\bar z,\cdot)\|_{L^1},
\]
and
\begin{align*}
|(F_{z})_*\psi-(F_{\bar z})_*\psi|_{BV^1}&  \le K_{\#} \Vert W( z, \cdot)-W(\bar z,\cdot)\Vert_{L^1}.
\end{align*}
Considering $\psi=\phi_{\bar z}$ in the above, the result follows.
\end{proof}

\begin{lemma}[Invariance of admissible set with small distortion]\label{F-invariante}
Assume that $\mathtt{var}_{p,L^1}(W)<\infty$ and that the hypotheses of Lemma \ref{F-Expanding} hold. Consider the set
\[
\mc A_{\bo M}:=\mc B_{s,M}\cap \tilde {\mc B}_{BV^1,M_1}\cap \tilde{\mc B}_{BV^2,M_2}\cap \mc M_{1,\Leb}
\] with $M_1$, $M_2>0$ as in Proposition \ref{Lem:UnifProponFiber}. Then, there is $\alpha_0>0$ sufficiently small,  $M>0$,  and $\bar n\in\N$ such that provided $|\alpha|<\alpha_0$
\[
\mc F^n(\mc A_{\bo M})\subset\mc A_{\bo M}\quad\quad\forall n\ge \bar n.
\]  
\end{lemma}

Let's introduce some notation that will be helpful in the computations to come. Define $\phi_0:=\phi$, $\phi_k:=\mc F^k\phi$, and denote $F^{n}_{\phi,z}=F_{\phi_{n-1},z}\circ...\circ F_{\phi_1,z}\circ F_{\phi_0,z}$ (similar notation can be defined with a general measure $\nu\in \mc M_{1,\Leb}$ rather than the density $\phi$).

We are going to use the following lemma that proves exponential loss of memory for the sequential composition of transfer operators $(F_{\phi,z})_*$, provided that the coupling strength is sufficiently small.
\begin{lemma}\label{Lem:MemLoss}
Under the hypotheses of Lemma \ref{F-Expanding}, there is $\alpha_0>0$ sufficiently small, $C>0,\,\lambda>0$ such that if $|\alpha|<\alpha_0$,  for any sequence $\{\nu_n\}_{n=1}^{\infty}$ in $ \mc M_{1,\Leb}$, and any $\psi\in \mc B_{BV^1}$ with $\int\psi=0$, and for a.e. $z\in[0,1]$
\begin{equation}\label{Eq:MemoryLoss}
\| [F_{\nu_{n},z}\circ...F_{\nu_2,z}\circ F_{\nu_1,z}]_*\psi\|_{BV^1}\le C e^{-\lambda n}\|\psi\|_{BV^1}.
\end{equation}
\end{lemma}
\begin{proof}
The proof is an application of Lemma 19 of \cite{CGT2}. More precisely, we are going to show that for sufficiently small $\alpha_0$, the sequence of operators $L_i:= (F_{\nu_{i},z})_*$ satisfies the assumptions of this lemma which provides the memory loss estimate \eqref{Eq:MemoryLoss}. Define $L_0:= (f)_*$, i.e.  the transfer operator of the uncoupled dynamics.

$L_i$'s are Markov operators on $(L^1(\T), \|\cdot \|_{L^1})$ and, for $\alpha_0>0$ sufficiently small, also leave $(\mc B_{BV^1},\|\cdot\|_{BV^1})$ invariant -- as consequence of Lemma \ref{F-Expanding} and Lemma \ref{Lem:BoundUnifExpMaps}. We take these spaces as our weak and strong space respectively when applying Lemma 19 \cite{CGT2}. We now only need to show that for these spaces conditions (ML1), (ML2), and (ML3) of Lemma 19 \cite{CGT2} are satisfied.

\smallskip
(ML1) follows from the fact that $\|L_i\psi\|_{L^1}\leq \|\psi\|_{L^1}$, and, provided $\alpha_0>0$ is small enough, Lemma \ref{F-Expanding} implies that the $L_i$ are all transfer operators of uniformly expanding maps with lower bound on the expansion and upper bound on the distortion independent of $i$. Arguing as in the proof of Lemma \ref{Lem:BoundUnifExpMaps}, we show that they all satisfy a common Lasota-Yorke inequality:
\begin{align*}
\|L_{i+m}...L_i\psi\|_{BV^1}&\le A\lambda_1^m\|\psi\|_{BV^1}+B\|\psi\|_{L^1}\\
\|L_0^m\psi\|_{BV^1}&\le A\lambda_1^m\|\psi\|_{BV^1}+B\|\psi\|_{L^1}
\end{align*}
for $A,B>0$ and $\lambda_1\in(0,1)$ independent of $i,m\in \N$.
\smallskip

To prove (ML2), notice that 
\begin{align*}
\| (L_i-L_0)\psi\|_{L^1}&=\sup_{\|g\|_\infty\le 1}\int g(L_i\psi-L_0\psi)\\
&=\sup_{\|g\|_\infty\le 1}\int (g\circ F_{\nu_i,z}-g\circ f)\psi\\
&=\sup_{\|g\|_\infty\le 1}\int \Phi'\psi
\end{align*}
where we defined
\[
\Phi(x):=\int_{f(x)}^{F_{\nu_i,z}(x)}g(s)ds 
\]
that satisfies $|\Phi(x)|\le \sup_{x\in\T}|F_{\nu_i,z}(x)-f(x)|=O(\alpha)$, implying
\[
\| (L_i-L_0)\psi\|_{L^1}\le O(\alpha) \|\psi\|_{BV^1}.
\]
This establishes that the operator norm of $(L_i-L_0)$ from the strong to the weak space is of order $\alpha$, and thus can be made arbitrarily small reducing $\alpha_0>0$.
\smallskip

Finally, (ML3) follows from classical results on exponential decay of correlations for uniformly expanding maps acting on functions of bounded variation (see e.g. \cite{NiloofarMarkLiverani})
\end{proof}

\begin{proof}[Proof of Lemma \ref{F-invariante}]
Pick $\phi\in \mc A_{\bold M}$.  We need to estimate
\begin{align*}
|(\mathcal{F}^n\varphi)_{z} - (\mathcal{F}^n\varphi)_{\bar{z}} |_{BV^1}             & =| (F^n_{\varphi,z})_{*}\varphi_{z} - (F^n_{\varphi,\bar{z}})_{*}\varphi_{\bar{z}} |_{BV^1}  \\
	& \leq \underbrace{| (F^n_{\varphi,z})_{*}\varphi_{z} - (F^n_{\varphi,z})_{*}\varphi_{\bar{z}} |_{BV^1}}_{\mytag{$A_{1}$}{termA1}} + \underbrace{| (F^n_{\varphi,z})_{*}\varphi_{\bar{z}} - (F^n_{\varphi,\bar{z}})_{*}\varphi_{\bar{z}} |_{BV^1}}_{\mytag{$A_{2}$}{termA2}}
\end{align*}
We start by giving a bound for the term (\ref{termA1}). Since $\phi_z-\phi_{\bar z}\in\mc B_{BV^1}$ and has zero integral, picking $\alpha$ sufficiently small and applying Lemma \ref{Lem:MemLoss}, 
\begin{align*}
| (F^n_{\varphi,z})_{*}(\varphi_{z}-\varphi_{\bar{z}}) |_{BV^1} & \le \| (F^n_{\varphi,z})_{*}(\varphi_{z}-\varphi_{\bar{z}})\|_{BV^1}\\
&\le Ce^{-\lambda n}\|\varphi_{z}-\varphi_{\bar{z}}\|_{BV^1}
\end{align*}
Using Lemma \ref{Lem:RelationL1andBV1}, we get 
\[
| (F^n_{\varphi,z})_{*}(\varphi_{z}-\varphi_{\bar{z}}) |_{BV^1}\le Ce^{-\lambda n}3|\varphi_{z}-\varphi_{\bar{z}}|_{BV^1}.
\]
Fix $\bar n$ so that for any $n\ge \bar n$, $Ce^{-\lambda n}3\le \tau\in(0,1)$. 

We now compute the term (\ref{termA2}). A telescopic sum and triangle inequality gives
\begin{align}
| (F^n_{\varphi,z})_{*}\varphi_{\bar{z}} - (F^n_{\varphi,\bar{z}})_{*}\varphi_{\bar{z}} |_{BV^1}&\le \sum_{j=0}^{n-1} |  (F_{\phi_{n-1}, z})_*...(F_{\phi_{j+1}, z})_*[(F_{\phi_{j}, z})_*-(F_{\phi_{j},\bar z})_*] (F_{\phi_{j-1},\bar z})_*...(F_{\phi_0,\bar z})_*\phi_{\bar z} |_{BV^1}\nonumber\\
&\le \sum_{j=0}^{n-1} Ce^{-\lambda (n-1-j)}  \|[(F_{\phi_{j}, z})_*-(F_{\phi_{j},\bar z})_*] (F_{\phi_{j-1},\bar z})_*...(F_{\phi_0,\bar z})_*\phi_{\bar z}\|_{BV^1}\label{Eq:Step1ML}\\
&\le O(\alpha)\sum_{j=0}^{n-1}Ce^{-\lambda (n-1-j)} \|(F_{\phi_{j-1},\bar z})_*...(F_{\phi_0,\bar z})_*\phi_{\bar z}\|_{BV^1}\label{Eq:Step2DistOp}\\
&\le  O(\alpha) \Vert W(z,\cdot)-W(\bar{z},\cdot)\Vert_{L^1}CM_1(1-e^{-\lambda})^{-1}\label{Eq3:Boundnormfiber}\\
&\le O(\alpha) \Vert W(z,\cdot)-W(\bar{z},\cdot)\Vert_{L^1}
\end{align}
where to obtain \eqref{Eq:Step1ML} we used Lemma \ref{Lem:MemLoss}; to obtain \eqref{Eq:Step2DistOp} we used Lemma \ref{DifDynSameDensityDERIVADA} and Lemma \ref{Lem:RelationL1andBV1}; to obtain  \eqref{Eq3:Boundnormfiber} we repeatedly applied Lemma \ref{Lem:BoundUnifExpMaps}. Notice that $O(\alpha)$ can be chosen uniformly in $n$.

So, we can see that for any $n\ge\bar n$
\begin{align}\label{desigualdadgrandisimaRESUELTA}
|(\mathcal{F}^n\varphi)_{z} - (\mathcal{F}^n\varphi)_{\bar{z}} |_{BV^1}   &  \leq \tau |  \varphi_{z} - \varphi_{\bar{z}} |_{BV^1}   + O(\alpha) \Vert W(z,\cdot)-W(\bar{z},\cdot)\Vert_{L^1}
\end{align}
so, taking the essential supremum on $z,\bar{z}\in B(\omega,r)$ and then the integral, we obtain
\begin{equation}
\mathtt{var}_{p,BV^1}(\mathcal{F}^n\varphi)\leq\tau\ \mathtt{var}_{p,BV^1}(\varphi) +O(\alpha)\ \mathtt{var}_{p,L^1}(W).
\end{equation}
Now, since $\mathtt{var}_{p,L^1}(W)<\infty$ 
\begin{equation*}
\mathtt{var}_{p,BV^1}(\mathcal{F}^n\varphi)\leq \tau\mathtt{var}_{p,BV^1}(\varphi) + B_{\#} 
\end{equation*}
where $B_\#$ is a constant independent of $\mc F^n(\phi)$.
It is now immediate to notice that 
\begin{align}
\Vert \mathcal{F}^n\varphi \Vert_{\textit{S}} & = \mathtt{var}_{p,BV^1}(\mathcal{F}^n\varphi)+\Vert \mathcal{F}^n\varphi \Vert_{"1"}\nonumber\\
& \leq\tau \ \mathtt{var}_{p,BV^1}(\varphi) + \bar B \label{Eq:}\\
& \leq\tau \, \Vert \varphi \Vert_{\textit{S}} + \bar {B}
\end{align}
Note that as $\varphi$ belongs to $\mathcal{B}_{\textit{S,M}}$ then $\Vert \varphi \Vert_{\textit{S}} \leq M$, so for $M\geq M_{0}:=\bar {B}(1-\tau)^{-1}$ the set $\mc A_{\bo M}$ is invariant.
\end{proof}

\subsubsection{Convexity and  pre-compactness of $\mc A_{ \bold M}$}
\begin{lemma}\label{espacioConvexo}
For any $M, M_1, M_2$, the set  $
\mc A_{\bo M}:=\mc B_{\textit{S,M}}\cap \tilde {\mc B}_{BV^1,M_1}\cap \tilde{\mc B}_{BV^2,M_2}\cap \mc M_{1,\Leb}$ is convex. 
\end{lemma}

\begin{proof}
Consider $\varphi, \psi\in \mc B_{\textit{S,M}}\cap\mc M_{1,\Leb}$ this means that $\Vert \varphi \Vert_{\textit{S}},\,\ \Vert \psi \Vert_{\textit{S}}\leq M$ and for any $\tau \in [0,1]$ we have that
\begin{align*}
\mathtt{osc}_{BV^1}(\tau \varphi +(1-\tau)\psi,\omega,r) 								& =\mathop{\mathrm{ess~sup}}\limits_{z,\bar{z}\in B(\omega,r)}| \tau \varphi_{z} +(1-\tau)\psi_{z} -\tau \varphi_{\bar{z}} -(1-\tau)\psi_{\bar{z}}|_{BV^1} \\ 
	& \leq \mathop{\mathrm{ess~sup}}\limits_{z,\bar{z}\in B(\omega,r)}|\tau \varphi_{z} -\tau \varphi_{\bar{z}} |_{BV^1} + \mathop{\mathrm{ess~sup}}\limits_{z,\bar{z}\in B(\omega,r)}|(1-\tau)\psi_{z} -(1-\tau)\psi_{\bar{z}} |_{BV^1} \\
	& = \tau \ \mathtt{osc}_{BV^1}(\varphi,\omega,r) + (1-\tau) \ \mathtt{osc}_{BV^1}(\psi,\omega,r)
\end{align*}
Taking the integral on $[0,1]$ with respect to $\omega$ and dividing by $r^p$ ($r>0$)
\begin{equation*}
\dfrac{1}{r^{p}}\int_{[0,1]}\mathtt{osc}_{BV^1}(\tau \varphi +(1-\tau)\psi,\omega,r) \ d\omega	 \leq \tau \ \dfrac{1}{r^{p}}\int_{[0,1]} \mathtt{osc}_{BV^1}(\varphi,\omega,r) \ d\omega  + (1-\tau) \ \dfrac{1}{r^{p}}\int_{[0,1]} \mathtt{osc}_{BV^1}(\psi,\omega,r) \ d\omega
\end{equation*}
So, by definition, we have that 
\begin{equation}\label{convexVar_C1}
\mathtt{var}_{p,BV^1}(\tau \varphi +(1-\tau)\psi)\leq \tau \ \mathtt{var}_{p,BV^1}( \varphi ) + (1-\tau) \ \mathtt{var}_{p,BV^1}(\psi).
\end{equation}
On the other hand, by the definition of Wasserstein-Kantorovich distance, we have 
\begin{align*}
\Vert \tau \varphi_{z} +(1-\tau)\psi_{z} \Vert_{W^1}
	& = \sup_{
\substack{\Lip(g)\leq 1 \\ \Vert g \Vert_{\infty}\leq 1 } } \left| \int g \cdot (\tau \varphi_{z} +(1-\tau)\psi_{z}) \ d \Leb \right| \\
	& \leq \sup_{
\substack{\Lip(g)\leq 1 \\ \Vert g \Vert_{\infty}\leq 1 } } \left| \int \tau \ g \cdot  \varphi_{z} \ d \Leb \right| + \sup_{
\substack{\Lip(g)\leq 1 \\ \Vert g \Vert_{\infty}\leq 1 } } \left| \int (1-\tau) \ g \cdot \psi_{z} \ d \Leb \right|\\
	& = \tau \ \Vert \varphi_{z} \Vert_{W^1} + (1-\tau) \ \Vert \psi_{z} \Vert_{W^1}
\end{align*}
If we take the integral over the interval $[0,1]$ on both sides of the inequality, we have
\begin{equation}\label{convexNorm1}
\Vert \tau \varphi +(1-\tau)\psi \Vert_{``1"} \leq \tau \ \Vert \varphi \Vert_{``1"} + (1-\tau) \ \Vert \psi \Vert_{``1"}
\end{equation}
Finally, putting together inequalities (\ref{convexVar_C1}) and (\ref{convexNorm1}), we obtain
\begin{align*}
\Vert \tau \varphi +(1-\tau)\psi \Vert_{\textit{S}} & = \mathtt{var}_{p,BV^1}(\tau \varphi +(1-\tau)\psi) +  \Vert \tau \varphi +(1-\tau)\psi \Vert_{``1"} \\
	& \le \tau \ \Vert \varphi \Vert_{\textit{S}} + (1-\tau) \ \Vert \psi \Vert_{\textit{S}}\\
        & \le  M.
\end{align*}
Now consider $\psi, \varphi\in \tilde{\mc B}_{BV^i,M_i}$ so for a.e. $z\in [0,1]$ we have $\psi_{z}, \varphi_{z}\in \mc B_{BV^i,M_i}$ that is $\Vert \psi_{z}\Vert_{BV^i} <M_i$ and $\Vert \varphi_{z}\Vert_{BV^i} <M_i$ then 
\[
\Vert \tau \varphi_{z} +(1-\tau)\psi_{z} \Vert_{BV^i} \le \tau \Vert \varphi_{z} \Vert_{BV^i} +(1-\tau)\Vert \psi_{z} \Vert_{BV^i}\le M_i. 
\]
So, for almost any $z\in [0,1]$ we have $\tau \varphi_{z} +(1-\tau)\psi_{z}$ belongs to $\mc B_{BV^i,M_i}$ then $\tau \varphi +(1-\tau)\psi$ belongs to $\tilde{\mc B}_{BV^i,M_i}$. Therefore, $\mc A_{\bo M}$ is convex.
\end{proof}

{

\begin{lemma}\label{RelativamenteCompacto}
The set $\mc A_{\bo M}$ is relatively compact in $\mathcal{B}_{w}$.
\end{lemma}
\begin{proof}
In \cite{Galatolo2018} the author considers the space $\mc B_\text{\textit{p-BV}}= \{ \varphi:[0,1]\times\T\rightarrow\R: \ \Vert \varphi \Vert_{\textit{p-BV}}<\infty \}$ with $\Vert \varphi \Vert_{\textit{p-BV}} = \Vert \varphi \Vert_{``1"} + \mathtt{var}_{\textit{p}, W^1}(\varphi)$ -- where $\mathtt{var}_{\textit{p}, W^1}$ is defined similarly to $ \mathtt{var}_{\textit{p}, BV^1}$, but with the $W^1$-norm instead of the $BV^{1}$ seminorm (see Section 3 in \cite{Galatolo2018}) -- and proves that any bounded ball in $\mc B_\text{\textit{p-BV}}$ is relatively compact in $\mc B^{}_{w}$. We are going to show that $\mc B_{\textit{S,M}}\subset \mc B_{\textit{p-BV}}$  and $\|\cdot\|_{\textit{p-BV}}\le \|\cdot\|_{\textit{S,M}}$, which will imply that also $\mathcal{B}_{\textit{S,M}}$ is relatively compact in $\mathcal{B}^{}_{w}$.

As the first step, we need to verify that the bounded ball as a subset of our strong space $\mathcal{B}_{\textit{S,M}}$ is a subset of the \textit{p-BV} space.  In fact, by Proposition \ref{psig}
\begin{equation*}
    W^{1}(  \varphi_{z},\varphi_{\bar{z}} ) = \sup_{
\substack{\Lip(g)\leq 1 \\ \Vert g \Vert_{\infty}\leq 1 } } \left| \int_\T g(s) (\varphi_{z} - \varphi_{\bar{z}})(s) \,ds \right| \leq \sup_{\substack{g \in \mc C^1(\mathbb{T}) \\ \|g\|_\infty \leq 1}} \int_{\mathbb{T}} g'(s)(\varphi_{z} - \varphi_{\bar{z}})(s)\,ds = | \varphi_{z} - \varphi_{\bar{z}}|_{BV^1}
\end{equation*}
this implies that  $\Vert  \varphi_{z} - \varphi_{\bar{z}} \Vert_{W^1} \leq | \varphi_{z} - \varphi_{\bar{z}}|_{BV^1}$. Taking the essential supremum on $z,\bar{z}\in B(\omega,r)$ and then the integral on $[0,1]$ on both sides of the inequality, we obtain
\begin{equation*}
\mathtt{var}_{\textit{p}, W^1}(\varphi) \leq \mathtt{var}_{\textit{p}, BV^1}(\varphi).
\end{equation*}
From the definitions $\|\cdot\|_{\textit{p-BV}} \le \|\cdot\|_{\textit{S}}$ and $\mc B_{\textit{S,M}}$ is a subset of $\mc B_{\textit{p-BV,M}}$, the ball in $\mc B_\text{\textit{p-BV}}$ centered at zero with radius $M$.
Therefore,  given a sequence of densities $\{ \varphi_{n} \}_{n\in\mathbb{N}}$ in $\mathcal{B}_{\textit{S,M}}$ such that $\Vert \varphi_{n}\Vert_{\textit{S}}\leq M$, $\Vert \varphi_{n}\Vert_{\textit{p-BV}}\leq \Vert \varphi_{n}\Vert_{\textit{S}}\leq M$ then,  applying Theorem 19 in \cite{Galatolo2018}, there exist  a subsequence $\{\varphi_{n_k}\}_{k\in\mathbb{N}}$ and $\varphi\in \mathcal{B}^{}_{w}$  such that $\Vert \varphi_{n_k} - \varphi\Vert_{``1"}\to 0$, which means that the set $\mathcal{B}_{\textit{S,M}}$ is relatively compact in $\mathcal{B}^{}_{w}.$ Finally, note that by  definition, the closure $\bar{\mathcal{B}}_{\textit{S,M}}$ is compact in $\mathcal{B}^{}_{w}.$ Since $\mc A_{\bo M}\subset \mathcal{B}_{\textit{S,M}}$ we have $\bar{\mc A}_M\subset \bar{\mathcal{B}}_{\textit{S,M}}$. As  $\bar{\mathcal{B}}_{\textit{S,M}}$ is compact then $\bar{\mc A}_M$ is a closed subset of a compact space, and hence compact, which implies that $\mc A_{\bo M}$ is relatively compact in $\mathcal{B}^{}_{w}.$
\end{proof}}

\subsubsection{The STO is Lipschitz in the Weak Norm}
\begin{lemma}\label{cont-STO} Assume that $W\in {L^{\infty}([0,1]^2)}$. Then the Self-Consistent transfer operator 
$\mathcal{F}: (\mathcal{B}_{\textit{S,M}}\cap \mc M_{1,\Leb},\Vert \cdot \Vert_{``1"})\to (\mathcal{B}_{\textit{S,M}}\cap \mc M_{1,\Leb},\Vert \cdot \Vert_{``1"})$ is Lipschitz continuous:
\begin{equation*}
\Vert \mathcal{F}\varphi - \mathcal{F}\psi \Vert_{``1"} \leq \mathrm{Lip}(\mathcal{F}) \, \Vert \varphi - \psi \Vert_{``1"}
\end{equation*}
for any $\phi,\psi\in \mathcal{B}_{\textit{S,M}}\cap \mc M_{1,\Leb},$ with  
\[\mathrm{Lip}(\mathcal{F})\leq \sup_{x\in\mathbb{T}} \vert f'(x) \vert +2 \, \alpha \|W\|_{L^{\infty}([0,1]\times [0,1])}\, \Vert h \Vert_{\mc C^1}.
\]
\end{lemma}
\begin{proof}

\begin{align*}
\Vert \mathcal{F}\varphi - \mathcal{F}\psi \Vert_{``1"} &  =  \int_{[0,1]} \Vert (F_{\phi,z})_{*}\phi_{z}-(F_{\psi,z})_{*}\psi_{z} \Vert_{W^1} \, dz \nonumber \\
& \leq \underbrace{\int_{[0,1]} \Vert (F_{\phi,z})_{*}\phi_{z}-(F_{\psi,z})_{*}\phi_{z} \Vert_{W^1}\, dz}_{\mytag{$C_1$}{termC1}} + \underbrace{\int_{[0,1]} \Vert (F_{\psi,z})_{*}\varphi_{z}-(F_{\psi,z})_{*}\psi_{z} \ \Vert_{W^1} \, dz}_{\mytag{$C_2$}{termC2}}.
\end{align*}

First we are going to estimate (\ref{termC1}):
\begin{eqnarray}
   \Vert (F_{\phi,z})_{*}\phi_{z}-(F_{\psi,z})_{*}\phi_{z} \ \Vert_{W^1} & =& \sup_{
\substack{\Lip(g)\leq 1 \\ \Vert g \Vert_{\infty}\leq 1 } } \left[ \int_\T g\cdot (F_{\phi,z})_{*}\phi_{z} \, d \Leb-\int_\T g\cdot (F_{\psi,z})_{*}\phi_{z} \, d \Leb \right] \nonumber \\
         & =& \sup_{
\substack{\Lip(g)\leq 1 \\ \Vert g \Vert_{\infty}\leq 1 } } \left[ \int_\T (g\circ F_{\phi,z}) \cdot \phi_{z}\,d \Leb-\int_\T (g\circ F_{\psi,z}) \cdot \phi_{z}\,d \Leb \right] \nonumber  \\
		& \leq & \sup_{
\substack{\Lip(g)\leq 1 \\ \Vert g \Vert_{\infty}\leq 1 }} \int_\T \Lip(g)\cdot \vert F_{\phi,z} - F_{\psi,z}\vert \cdot \phi_{z}\,d \Leb  \nonumber  \\
		& \leq & \int_\T \vert F_{\phi,z} - F_{\psi,z}\vert \cdot \phi_{z} \, d \Leb. \nonumber 
\end{eqnarray}

\noindent
Moreover, since $h\in \mathcal{C}^{3}(\mathbb{T}\times\mathbb{T},\mathbb{R})$, for every $x\in\mathbb{T}$ 
\begin{equation*}
\int_{\mathbb{T}} h(x,y) (\phi_{z}(y)-\psi_{z}(y)) \, dy  \leq \Vert h \Vert_{\mc C^1} \Vert \phi_{z}-\psi_{z}\Vert_{W^1}.
\end{equation*}
Using the previous inequality and the Hölder inequality, we obtain
\begin{align*}
\left| F_{\phi,z}(x)- F_{\psi,z}(x)\right| & = \alpha \left| \int_{0}^{1} dz' W(z,z') \int_{\mathbb{T}} dy \ h(x,y)  (\phi_{z'}(y)-\psi_{z'}(y)) \right| \\
	& \leq \alpha \left| \int_{0}^{1} dz' W(z,z') \Vert h \Vert_{\mc C^1} \Vert \phi_{z'}-\psi_{z'}\Vert_{W^1} \right|  \\
	& \leq \alpha \, \Vert W(z,\cdot) \Vert_{L^{\infty}([0,1])} \Vert h \Vert_{\mc C ^1}\left| \int_{0}^{1} dz' \Vert \phi_{z'}-\psi_{z'}\Vert_{W^1} \right|  \\
	& = \alpha \,\Vert W(z,\cdot) \Vert_{L^{\infty}([0,1])} \Vert h \Vert_{\mc C^1}\Vert \phi - \psi \Vert_{``1"}. 
\end{align*}
Recalling that $\phi_{z}$ is a probability density,  we have
\begin{align*}
   \int_0^1 \Vert (F_{\phi,z})_{*}\phi_{z}(x)-(F_{\psi,z})_{*}\phi_{z}(x) \Vert_{W^1}\, dz & \leq  \alpha \, \Vert W \Vert_{L^{\infty}([0,1]^2)} \Vert h \Vert_{\mc C^1}\Vert \phi - \psi \Vert_{``1"}
\end{align*} 
where we use that $\int_{0}^{1} dz \Vert W(z,\cdot) \Vert_{L^{\infty}([0,1])}=\Vert W \Vert_{L^{1}([0,1],L^{\infty}([0,1]))}.$

For term (\ref{termC2}) we use  again the duality property of the operator
\begin{align*}
   \Vert (F_{\psi,z})_{*}\phi_{z}-(F_{\psi,z})_{*}\psi_{z} \ \Vert_{W^1} & = \sup_{
\substack{\Lip(g)\leq 1 \\ \Vert g \Vert_{\infty}\leq 1 } } \left[ \int (g\circ F_{\psi,z}) \cdot \phi_{z}d \Leb-\int (g\circ F_{\psi,z}) \cdot \psi_{z}d \Leb \right] \\
		& \leq \sup_{
\substack{\Lip(g)\leq 1 \\ \Vert g \Vert_{\infty}\leq 1 } } \int \Lip(g\circ F_{\psi,z})\cdot \vert \phi_{z}-\psi_{z}\vert \, d \Leb \nonumber  \\
		& \leq  \int \, \sup\vert F'_{\psi,z} \vert \, \vert \phi_{z}-\psi_{z}\vert \, d \Leb.
\end{align*}

Note that for all $x\in \mathbb{T}$ and $\psi\in \mc B_{\textit{S,M}}\cap\mc M_{1,\Leb}$ the derivative of  $F_{\psi,z}:\mathbb{T}\to\mathbb{T}$ is uniformly bounded (Lemma \ref{F-Expanding})
so,
\begin{align*}
   \Vert (F_{\psi,z})_{*}\phi_{z}-(F_{\psi,z})_{*}\psi_{z} \ \Vert_{W^1} & \leq (\kappa + \alpha\,\Vert h\Vert_{\mc C^1}\, \|W(z,\cdot)\|_{L^{1}([0,1])}) \, \Vert \phi_{z}-\psi_{z}\Vert_{W^1}
\end{align*}
Term (\ref{termC2}) is then bounded by
\begin{equation}\label{segundaestimative}
\int_{[0,1]} \Vert (F_{\psi,z})_{*}\phi_{z}-(F_{\psi,z})_{*}\psi_{z} \ \Vert_{W^1} \, dz \leq (\kappa + \alpha\,\Vert h\Vert_{\mc C^1}\, \|W\|_{L^{\infty}([0,1],L^{1}([0,1]))}) \, \Vert \phi-\psi\Vert_{``1"}
\end{equation}
Finally, if we put together all of the above, we obtain

\begin{equation*}
\Vert \mathcal{F}\varphi - \mathcal{F}\psi \Vert_{``1"} \leq (\kappa + 2\, \alpha \, \Vert h\Vert_{\mc C^1} \|W\|_{L^{\infty}([0,1]\times [0,1])}) \,\Vert \varphi_1 - \varphi_2 \Vert_{``1"}
\end{equation*}
\end{proof}

\subsubsection{Existence of a fixed point density for an iterate of $\mc F^n$}
\begin{proposition}\label{Prop:periodicOrbit}
Under the hypotheses of Lemma \ref{F-invariante}, and for an $n\in \N$ such that  $\mathcal{F}^n(\mc A_{\bo M})\subset \mc A_{\bo M}$, $\mc F^n$ has a fixed point $\varphi^{*}$ in the closure of $\mc A_{\bo M}$ in $\mc B_w$.
\end{proposition}
\begin{proof}
The structure of the proof is as follows. In Lemma~\ref{F-invariante}, we showed that the set $\mc A_{\bo M}$ is invariant under the action of $\mathcal{F}^n$ for suitable parameters $\bo M$ and for $n\in\N$ sufficiently large. Since $\mc A_{\bold M} \subset \mathcal{B}_{\textit{S,M}} \cap \mc M_{1,\Leb}$, and $\mathcal{F}$ is Lipschitz continuous on $\mathcal{B}_{\textit{S,M}}$ with respect to the weak norm $\|\cdot\|_{``1"}$ (see Lemma~\ref{cont-STO}), we conclude that the restriction $\mathcal{F}^n|_{\mc A_{\bo M}}: \mc A_{\bo M} \to \mc A_{\bo M}$ is Lipschitz with respect to the distance induced by the norm $\|\cdot\|_{``1"}$.

Next, in Theorem~\ref{RelativamenteCompacto}, we proved that $\mc A_{\bo M}$ is relatively compact in the weak space $\mathcal{B}^{}_{w}$, and in Theorem~\ref{espacioConvexo} we established that $\mc A_{\bo M}$ is convex. Consequently, its closure $\overline{\mc A}_{M}$ in $\mathcal{B}^{}_{w}$ is both compact and convex.

As previously noted, the map $\mathcal{F}|_{\mc A_{\bo M}}$ is Lipschitz and hence uniformly continuous with respect to the weak norm. By the classical extension theorem for uniformly continuous functions (See e.g., \cite{rudin1964principles}, Theorem 4.19), there exists a unique continuous extension $\overline{\mathcal{F}}: \overline{\mc A}_{M} \to \mathcal{B}^{}_{w}$ satisfying $\overline{\mathcal{F}}|_{\mc A_{\bo M}} = \mathcal{F}$. Moreover, since $\mathcal{F}^n(\mc A_{\bo M}) \subset \mc A_{\bo M}$, continuity ensures that $\overline{\mathcal{F}}^n(\overline{\mc A}_{M}) \subset \overline{\mc A}_{M}$. Therefore, Schauder’s Fixed Point Theorem guarantees the existence of a fixed point $\varphi^\star \in \overline{\mc A}_{M}$. 
\end{proof}

\begin{corollary}\label{Cor:C2normFixed}
For a.e. $z\in[0,1]$, $\phi^\star_{z}\in \mc C^2(\T,\R)$ and $\mathop{\mathrm{ess~sup}}\limits_{z\in [0,1]}\|\phi^\star_z\|_{\mc C^2}<\infty$.
\end{corollary}
\begin{proof}
Calling $\phi_i^\star:=\mc F^i\phi^\star$ for $1\le i\le n-1$, one has that by definition of $\mc F$
\[
\phi^\star_z=\mc F^n\phi^\star=(F_{\phi_{n-1}^\star,z})_*(F_{\phi^\star_{n-2},z})_*...(F_{\phi^\star,z})_*\phi^\star_z.
\]
As a consequence of Lemma \ref{F-Expanding}, $\phi^\star_z$ is the fixed point of a uniformly expanding map with bounded $\mc C^3$ norm and it s known that this implies $\phi^\star_z\in \mc C^2$ (see e.g., \cite{krzyzewski1977some}, Theorem 1). Since $F_{\phi_{n-1}^\star,z}\circ F_{\phi^\star_{n-2},z}\circ...\circ F_{\phi^\star,z}$ is a uniformly expanding map with bounded $\mc C^3$ norm uniformly for a.e. $z$, the bound on the $\mc C^2$ norm of $\phi^\star_z$ is uniform for a.e. $z$, i.e.  $\mathop{\mathrm{ess~sup}}\limits_{z\in [0,1]}\|\phi^\star_z\|_{\mc C^2}<\infty$.
\end{proof}

\subsubsection{Existence, stability, and uniqueness of a fixed point for $\mc F$}

Notice that Proposition \ref{Prop:periodicOrbit} shows that $\mc F$ has a periodic orbit  $\{\phi_*,\mc F\phi_*,...,\mc F^{n-1}\phi_*\}\subset \mc A_{\bo M}$ of some period $n\ge 1$. We are going to prove now that $\mc F$ is a strict contraction with respect to a suitable metric and that $\phi_*$ must be a fixed point. Let's highlight that the metric with respect to which $\mc F$ is a contraction is not complete, therefore one cannot deduce existence of the fixed point from Banach fixed point theorem.

\smallskip
Let $\mc V_a$ be the cone of $a$-$\log$-Lipschitz functions on $\T$ taking real values \footnote{ $\mathcal{V}_{a}:= \left\{\psi \in \mathcal{C}(\mathbb{T},\mathbb{R}): \frac{\psi(x)}{\psi(y)}\le e^{a|x-y|} \right\}$.}, and define 
\[
\tilde{\mc V}_a:=\left\{\nu \in \mc M_{\Leb}([0,1]\times\T):\,\frac{d\nu_z}{dLeb}\in\mc V_a \mbox{ for a.e. }z\right\}.
\]
Denote by $\theta_a:\mc V_a\times\mc V_a\rightarrow \R\cup\{\infty\} $ the projective Hilbert metric on $\mc V_a$. For $\nu,\,\nu' \in \tilde{\mc V}_{a}$, define the (projective) metric
\[
\tilde \theta_a(\nu,\nu'):=\mathop{\mathrm{ess~sup}}\limits_{z\in [0,1]}\,  \theta_a\left(\frac{d\nu_z}{dLeb},\frac{d\nu'_z}{dLeb}\right).
\]

\begin{proposition}\label{Prop:Contraction}
There are $\alpha_1>0$  sufficiently small and $\eta\in[0,1)$, such that for all $|\alpha|<\alpha_1$  there is $a>0$ such that $\mathcal F$ sends $\tilde{\mc V}_a$ into $\tilde{\mc V}_{\eta a}$. Furthermore, suppose that 
$\phi,\,\psi\in \mc M_{1,\Leb}\cap \tilde{\mc V}_a$  and for almost every $z\in[0,1]$, $\psi_{z} \in \mathcal{C}^{2}(\mathbb{T})$ with $\mathop{\mathrm{ess~sup}}_{z\in[0,1]}\|\psi_{z}\|_{\mc C^2}<\infty$. Then, if $\alpha$ is sufficiently small there is $\gamma\in[0,1)$ such that 
\[
\tilde \theta_a(\mc F \phi, \mc F\psi)\le \gamma \tilde \theta_a(\phi,\psi).
\]
\end{proposition}

\begin{proof}
Pick $\alpha$ small enough so that  $F_{\nu,z}$ is a map with minimal expansion bounded by $|F_{\nu,z}'|\ge \sigma>1$ and bounded distortion $K>0$ uniform in $\nu$ and $z$ (applying Lemma \ref{F-Expanding}).  All the  maps $\{F_{\nu,z}\}_{\nu,z}$, map $\mc V_a$ to $\mc V_{\eta a}$ for some $a>0$ and $\lambda\in(0,1)$ depending on $\sigma$ and $K$, and since $(\mc F\nu)_z=(F_{\nu,z})_*\nu_z$ the first claim follows.

Let $\psi,\,\phi\in\mc M_{1,\Leb}([0,1]\times \T)$ be absolutely continuous. Fix a $z\in[0,1]$ and suppose $\varphi_{z}, \psi_{z}\in \mathcal{V}_{a}$. We want to estimate the distance in the Hilbert projective metric of the action of the STO restricted to the fixed fiber $\{z\}\times\mathbb{T}$ on these two densities: for a.e. $z\in[0,1]$
\begin{align*}
	\theta_{a}((\mathcal{F}\varphi)_{z},(\mathcal{F}\psi)_{z}) & = \theta_{a}((F_{\varphi,z})_{*}\varphi_{z},(F_{\psi,z})_{*}\psi_{z}) \\
				  & \leq \theta_{a}((F_{\varphi,z})_{*}\varphi_{z},(F_{\varphi,z})_{*}\psi_{z})+\theta_{a}((F_{\varphi,z})_{*}\psi_{z},(F_{\psi,z})_{*}\psi_{z})\\
				  & \leq \lambda \theta_{a}(\varphi_{z},\psi_{z})+C\|F_{\varphi,z}-F_{\psi,z}\|_{\mc C^2}
\end{align*}
where to obtain the last inequality we used that $(F_{\varphi,z})_{*}:({\mc V}_a,\theta_a)\rightarrow (\mc V_{\eta a},\theta_a)$ is a  contraction, with contraction rate $\lambda\in[0,1)$ depending only on $\eta$.  For the remaining term, we apply Proposition A.7 from \cite{Tanzi2022}, which provides a comparison between the Hilbert metric and the $\mathcal{C}^{2}$-metric where $C$ is a constant that depends on upper bounds of $\Vert F_{\varphi,z} \Vert$ and $\Vert F_{\psi,z} \Vert$, on lower bounds for $\inf \vert F'_{\varphi,z} \vert$ and $\inf \vert F'_{\psi,z} \vert$, and on $\Vert \psi_z \Vert_{\mc C^2}$; therefore $C$  can be chosen uniformly in $\phi$, $z$, and $\psi$ once $\|\psi_z\|_{\mc C^2}$ is uniformly bounded. This estimate requires $\psi_z \in \mathcal{C}^{2}(\mathbb{T})$ to control the distortion.

In addition, using Lemma \ref{VariationOfDensity} and Lemma \ref{C0ANDhm} we obtain

\begin{equation*}\label{relationC2HILBERT}
\|F_{\varphi,z}-F_{\psi,z}\|_{\mc C^2}\leq C_{\#} \alpha \mathop{\mathrm{ess~sup}}\limits_{z'\in[0,1]} \theta_{a}(\varphi_{z'},\psi_{z'}).
\end{equation*}
Using this last inequality, we have that
\begin{equation*}
\theta_{a}((\mathcal{F}\varphi)_{z},(\mathcal{F}\psi)_{z}) \leq \lambda \, \theta_{a}(\varphi_{z},\psi_{z}) +C_{\#} \alpha \mathop{\mathrm{ess~sup}}\limits_{z'\in[0,1]} \theta_{a}(\varphi_{z'},\psi_{z'})
\end{equation*}
Taking the essential supremum of both sides we get the inequality
\begin{equation*}
\mathop{\mathrm{ess~sup}}\limits_{z\in [0,1]}\theta_{a}((\mathcal{F}\varphi)_{z},(\mathcal{F}\psi)_{z}) \leq (\lambda+C_{\#}\alpha)  \mathop{\mathrm{ess~sup}}\limits_{z\in[0,1]} \theta_{a}(\varphi_{z},\psi_{z}).
\end{equation*}
So, there exists $\alpha_{\star}>0$ such that for all
weak coupling $\alpha<\alpha_{\star}$ where $\alpha_{\star}=C_{\#}^{-1}(1-\lambda)$ we have that  $\gamma=\lambda+C_{\#}\alpha<1.$ So, we obtain 
\begin{equation*}
\mathop{\mathrm{ess~sup}}\limits_{z\in [0,1]}\theta_{a}((\mathcal{F}\varphi)_{z},(\mathcal{F}\psi)_{z}) \leq \gamma  \mathop{\mathrm{ess~sup}}\limits_{z\in [0,1]} \theta_{a}(\varphi_{z},\psi_{z}).
\end{equation*}
\end{proof}

\begin{proposition}\label{Pro:fixedPointCone}
Under the assumptions of Lemma \ref{F-Expanding}, suppose that for $n\in \N$, $\mc F^n$ has a fixed point $\nu^\star\in\mc M_{1,\Leb}([0,1]\times\T)$ with $\nu^\star_z$ absolutely continuous for almost every $z\in[0,1]$. Then there is $a>0$ such that $\nu\in \tilde{\mc V}_a$.
\end{proposition}

\begin{proof}
For a.e. $z\in[0,1]$, call $\phi^\star_z:=\frac{d\nu_z^\star}{d\Leb}$ and $\phi^\star$ the density of $\nu^\star$. Call $\phi_i^\star:=\mc F^i\phi^\star$ for $1\le i\le n-1$. Then, for a.e. $z$ the density $\phi^\star_z$ is a fixed point of the uniformly expanding map $F_{\phi_{n-1}^\star,z}\circ...\circ F_{\phi_1^\star,z}\circ F_{\phi^\star,z}$:
\[
\phi_z^\star=(\mc F^n\phi^\star)_z=(F_{\phi^\star_{n-1},z})_*...(F_{\phi^\star,z})_*\phi_z^\star.
\]
Therefore, it is well known that $\phi_z^\star$ belongs to $\mc V_a$ for some $a$ depending on the minimal expansion and distortion of  $F_{\phi^\star,z}$ (also compare with Corollary \ref{Cor:SmoothnessFixedPoint}). Since for $z$ in a set of full measure, the minimum expansion and the distortion of  $F_{\phi^\star,z}$ are respectively uniformly lower bounded and upper bounded, $a$ can be chosen uniformly on (a.e.) $z$.
\end{proof}

Finally we can combine all of the above lemmas and propositions to prove our main result.

\begin{proof}[Proof of Theorem \ref{main}]
By Proposition \ref{Pro:fixedPointCone} there exists $a>0$ such that the fixed point $\phi^{*}$ of $\mc F^n$ is in $\tilde{\mathcal{V}}_{a}$, and therefore $\mc F^i\phi^\star\in \tilde{\mc V}_a$ for every $1\le i\le n-1$. Furthermore, by Corollary \ref{Cor:C2normFixed},  $\mathop{\mathrm{ess~sup}}_{z\in[0,1]}\|(\mc F^i\phi^\star)_z\|_{\mc C^2}<\infty$ for every $0\le i\le n-1$ and we can apply Proposition \ref{Prop:Contraction} to obtain
\[
\tilde\theta_a\left(\phi^\star,\mc F\phi^\star\right)=\tilde\theta_a\left(\mc F^{n}\phi^\star,\mc F^{n+1}\phi^\star\right)\le \lambda^{n}\tilde\theta_a\left(\phi^\star,\mc F\phi^\star\right)
\]
which implies $\theta_a\left(\phi^\star,\mc F\phi^\star\right)=0$, and therefore $\phi^\star=\mc F\phi^\star$, proving that $\phi^\star$ is a fixed point for $\mc F$.

We now prove stability of $\phi^\star$. Given $\epsilon>0$, consider $\phi$  a small $\mathcal{C}^1$ perturbation of $\phi^{*}$, that is for a.e. $z\in [0,1]$ $\|\varphi^{*}_{z}- \varphi_{z}\Vert_{\mathcal{C}^1} \leq \delta$ with 
\[
\delta \le \min \left\{ \frac{1}{4ae^{a/2}+6e^{a}}\varepsilon,\frac{e^{-a/2}}{2} \right\}.
\]
By Lemma 4.6 in \cite{CastorriniMatteo} we have that $\varphi_{z}\in \mathcal{V}_{a+\varepsilon}$ for a.e. $z\in[0,1]$ and therefore $\phi\in \tilde{\mathcal{V}}_{a+\varepsilon}$. Recall that we can consider $a>0$ sufficiently large such that $\mc F$ sends $\tilde{\mathcal{V}}_{a}$ strictly inside itself; it is not hard to prove that also $\mathcal{F}\tilde{\mathcal{V}}_{a+\varepsilon} \subset \tilde{\mathcal{V}}_{a+\varepsilon}$ if $\epsilon>0$ is sufficiently small. 
Since $\varphi^\star_z$ belongs to $\mathcal{C}^{2}(\mathbb{T})$ for almost every $z \in [0,1]$ we can again apply Proposition \ref{Prop:Contraction} and obtain
\begin{align*}
	\tilde{\theta}_{a+\varepsilon}(\mathcal{F}\phi, \phi^{*}) & = \mathop{\mathrm{ess~sup}}\limits_{z\in [0,1]}\,  \theta_{a+\varepsilon}\left({(\mathcal{F}\phi)_{z}},\phi^{*}_z\right) \\
	& = \mathop{\mathrm{ess~sup}}\limits_{z\in[0,1]}\,  \theta_{a+\varepsilon}\left((\mathcal{F}\phi)_{z},(\mathcal{F}\phi^{*})_z\right)\\
	& \leq \gamma \, \mathop{\mathrm{ess~sup}}\limits_{z\in [0,1]}  \theta_{a+\varepsilon}\left(\phi_{z},\phi^{*}_z\right).
\end{align*}
and by iteration
\[
\tilde{\theta}_{a+\varepsilon}(\mathcal{F}^{n}\phi, \phi^{*})\leq \gamma^{n} \, \tilde{\theta}_{a+\varepsilon}\left(\phi,\phi^{*}\right).
\]
Applying Lemma \ref{C0ANDhm}  the result follows.

The uniqueness of the fixed density in the weak coupling regime follows analogously. Suppose that there are two fixed density denoted  by $\phi^\star,\,\hat{\varphi}\in \tilde{\mc V}_a$ for our Self Consistent Transfer Operator. Using the $\tilde{\theta}_{a}$ contraction we have 
$$\tilde{\theta}(\hat{\varphi},\varphi_{*})=\tilde{\theta}(\mathcal{F}\hat{\varphi},\mathcal{F}\varphi^{*})\leq \gamma \, \tilde{\theta}(\hat{\varphi},\varphi^{*})$$
As we know $0<\gamma<1$ so, $\tilde{\theta}(\hat{\varphi},\varphi_{*})=0$, by definition of $\tilde{\theta}_{a}$ implies that for a.e. $z\in [0,1]$ we have $\theta_{a}(\tilde{\varphi}_{z},\varphi^{*}_{z})=0$  and since $\tilde{\varphi}_{z}$ and $\varphi^{*}_{z}$ are both densities,  $\hat{\varphi}=\varphi^{*}$. 

\end{proof}

\subsection{Lipschitz Continuity of the STO: Theorem \ref{convergence_STO}}

\begin{proof}[Proof of Theorem \ref{convergence_STO}]
Consider
\begin{align*}
\Vert \, \tilde{\mc F}\tilde\phi - \Fvarphi \, \Vert_{``1"} & =  \int_0^1 \Vert (\tilde F_{\tilde\phi,z})_{*}\tilde\phi_{z}-(F_{\phi,z})_{*}\phi_z \ \Vert_{W^1} \ dz  \\
& \leq \underbrace{\int_0^1 \Vert (\tilde F_{\tilde\phi,z})_{*}\tilde\phi_{z}-(F_{\phi,z})_{*}\tilde \phi_z \Vert_{W^1}\ dz}_{\mytag{$D_{1}$}{termD1}} +\underbrace{\int_0^1 \Vert (F_{\phi,z})_{*}\tilde \phi_z-(F_{\phi,z})_{*}\phi_{z} \ \Vert_{W^1} \, dz}_{\mytag{$D_{2}$}{termD2}}
\end{align*}
First, we estimate (\ref{termD1}). For this notice that
\begin{align*}
   \Vert (\tilde F_{\tilde \phi,z})_{*}\tilde\phi_z-(F_{\phi,z})_{*}\tilde\phi_z \ \Vert_{W^1} & = \sup_{
\substack{\Lip(g)\leq 1 \\ \Vert g \Vert_{\infty}\leq 1 } } \left[\int_\T g\cdot (\tilde F_{\tilde \phi,z})_{*}\tilde\phi_z d \Leb-\int_\T g\cdot (F_{\phi,z})_{*}\tilde\phi_z \, d \Leb \right] \\
         & = \sup_{
\substack{\Lip(g)\leq 1 \\ \Vert g \Vert_{\infty}\leq 1 } } \left[ \int_\T (g\circ \tilde F_{\tilde \phi,z} ) \cdot \tilde\phi_zd \Leb-\int_\T (g\circ F_{\phi,z}) \cdot \tilde\phi_z \,d \Leb \right] \\
		& \leq \sup_{
\substack{\Lip(g)\leq 1 \\ \Vert g \Vert_{\infty}\leq 1 } } \int_\T \Lip(g)\vert \tilde F_{\tilde \phi,z} - F_{\phi,z}\vert \cdot \tilde\phi_z \, d \Leb  \\
		& \leq  \int_\T \vert\tilde F_{\tilde \phi,z} - F_{\phi,z}\vert \cdot \tilde\phi_z \,d \Leb. 
\end{align*}
In the integrand above we estimate
\[
\vert \tilde F_{\tilde \phi,z}(x) -  F_{\phi,z}(x) \vert \leq \vert \tilde F_{\tilde \phi,z}(x) - \tilde F_{\phi,z}(x) \vert + \vert \tilde F_{\phi,z}(x) - F_{\phi,z}(x) \vert
\]
the first term is bounded by
\begin{align*}
\vert \tilde F_{\tilde \phi,z}(x) - \tilde F_{\phi,z}(x) \vert & = \left| \alpha \int_{0}^{1}dz'\int_{\mathbb{T}} \, dy\, \tilde W (z,z') (\tilde \phi_{z'}(y)- \phi_{z'}(y)) \, h(x,y) \right| \\
& \leq  \alpha \int_{0}^{1} dz'\, \vert \tilde W (z,z')\vert \left| \int_{\mathbb{T}} dy\, (\tilde \phi_{z'}(y)- \phi_{z'}(y)) \, h(x,y) \right|\\
& \leq \alpha \int_{0}^{1} dz'\, \vert \tilde W (z,z')\vert \Vert \partial_{1}h(x,\cdot)\Vert_{\infty}\Vert \tilde \phi_{z'}- \phi_{z'} \Vert_{W^1}\\
& \leq \alpha \, \Vert h\Vert_{\mathcal{C}^1} \Vert \tilde{W}(z,\cdot)\Vert_{L^{\infty}([0,1])} \int_{0}^{1} dz'\, \Vert \tilde \phi_{z'}- \phi_{z'}\Vert_{W^1} \\
& \leq \alpha \, \Vert h\Vert_{\mathcal{C}^1} \Vert \tilde{W}(z,\cdot)\Vert_{L^{\infty}([0,1])} \, \Vert \tilde \phi - \phi \Vert_{``1"}
\end{align*}
and the second term is bounded by
\begin{align*}
	\vert \tilde F_{\phi,z}(x) - F_{\phi,z}(x) \vert & = \left| \alpha \int_{0}^{1}dz'\int_{\mathbb{T}} \, dy\, \phi_{z'}(y) (\tilde W (z,z') - W(z,z'))  \, h(x,y) \right| \\
	& \leq \alpha \,\int_{0}^{1} dz'\, \left| \int_{\mathbb{T}}  dy\, \phi_{z'}(y) \, h(x,y)\right| \vert \tilde W (z,z') - W(z,z')) \vert \\
	& \leq \alpha \,\Vert \tilde W (z,\cdot) - W(z,\cdot)\Vert_{L^{1}([0,1])} \Vert h\Vert_{\infty}\sup_{z'\in [0,1]}\Vert \varphi_{z'}\Vert_{L^1(\T)}\\
	& \leq \alpha \, \Vert h\Vert_{\infty} \Vert \tilde W (z,\cdot) - W(z,\cdot)\Vert_{L^{1}([0,1])}
\end{align*}

Thus, the term (\ref{termD1}) is bounded by
\begin{equation}\label{primeraestimativaConv}
\int_0^1\Vert (\tilde F_{\tilde \phi,z})_{*}\tilde\phi_z-(F_{\phi,z})_{*}\tilde\phi_z \Vert_{W^1}\ dz \leq \alpha \,  \Vert h \Vert_{\mc C^1} \left( \Vert \tilde W \Vert_{L^1([0,1],L^\infty([0,1]))} \Vert\tilde \phi- \phi \Vert_{``1"} + \Vert \tilde W  - W \Vert_{L^1([0,1]^2)}  \right)
\end{equation}
where we used that $\int_0^1\, dz \, \Vert \tilde W (z,\cdot) \Vert_{L^{\infty}([0,1])} = \Vert \tilde W \Vert_{L^1([0,1],L^\infty([0,1]))}$ and $\int_0^1\, dz \, \Vert\tilde W (z,\cdot) - W(z,\cdot) \Vert_{L^{1}([0,1])} = \Vert\tilde W  - W \Vert_{L^1([0,1]^2)}$.

Now, let's move on to the second term (\ref{termD2})

\begin{align*}
   \Vert (F_{\phi,z})_{*}\tilde\phi_z(x)-(F_{\phi,z})_{*}\phi_{z}(x) \ \Vert_{W^1} &  = \sup_{
\substack{\Lip(g)\leq 1 \\ \Vert g \Vert_{\infty}\leq 1 } } \left[ \int_\T g\cdot (F_{\phi,z})_{*}\tilde\phi_z \, d \Leb-\int_\T g\cdot (F_{\phi,z})_{*}\phi_{z} \, d \Leb \right]\\
         & = \sup_{
\substack{\Lip(g)\leq 1 \\ \Vert g \Vert_{\infty}\leq 1 } } \left[ \int_\T g\circ F_{\phi,z}  \cdot (\tilde\phi_z-\phi_z) d \Leb \right] \\
		& \leq  \sup_{
\substack{\Lip(g)\leq 1 \\ \Vert g \Vert_{\infty}\leq 1 } } \Lip(g\circ F_{\phi,z}) \, W^{1}( \tilde\phi_z,\phi_{z})  \\
		& \leq  \sup_{x\in\T}\vert F'_{\phi,z}(x) \vert \, W^{1}( \tilde\phi_z,\phi_{z}). 
\end{align*}
Recall that $\vert F'_{\phi,z} \vert$ is uniformly bounded in $\phi$ and $z$ (see Lemma \ref{F-Expanding}), so we have
\begin{align*}
\Vert (F_{\phi,z})_{*}\tilde\phi_z(x)-(F_{\phi,z})_{*}\phi_{z}(x) \ \Vert_{W^1} &  \leq (\kappa + \alpha\,\Vert h\Vert_{\mc C^1}\, \|W\|_{L^{\infty}([0,1],L^{1}([0,1]))}) \, \Vert \tilde\phi_z - \phi_{z}\Vert_{W^1}.
\end{align*}
Thus, term (\ref{termD2}) is bounded by
\begin{equation}\label{segundaestimativeConv}
\int_{[0,1]} \Vert (F_{\phi,z})_{*}\tilde\phi_z(x)-(F_{\phi,z})_{*}\phi_{z}(x) \ \Vert_{W^1} \, dz \leq (\kappa + \alpha\,\Vert h\Vert_{\mc C^1} \|W\|_{L^{\infty}([0,1],L^{1}([0,1]))})  \, \Vert\tilde \phi- \phi \, \Vert_{``1"}.
\end{equation}
Finally, combining using (\ref{primeraestimativaConv}) and (\ref{segundaestimativeConv}) we obtain

\begin{equation*}
\Vert \tilde{\mc F}\tilde \phi - \Fvarphi \, \Vert_{``1"}  \leq C \Vert\tilde \phi- \phi \ \Vert_{``1"} + \alpha \, \Vert h \Vert_{\mc C^1} \Vert \tilde W  - W \Vert_{L^1([0,1]^2)} 
\end{equation*}

where 
\begin{align*}
	C &=  \kappa+\alpha \, \Vert h \Vert_{\mc C^1}(\Vert \tilde W \Vert_{{L^1}([0,1],L^\infty([0,1]))}+\Vert W \Vert_{L^{\infty}([0,1],{L^{1}}([0,1]))})\\
	& \le \kappa+\alpha \, \Vert h \Vert_{\mc C^1}(\Vert \tilde W \Vert_{{L^{\infty}}([0,1],L^\infty([0,1]))}+\Vert W \Vert_{L^{\infty}([0,1],{L^{\infty}}([0,1]))}).
\end{align*}
\end{proof}

\begin{appendices}
\section{Some Useful Estimates}
\begin{lemma}\label{Lem:BoundUnifExpMaps}
Suppose that $F\in \mc C^3(\T,\T)$ and $|F'(x)|\ge\sigma>1$ and $\|F\|_{\mc C^3}<K$. Then there are $M_1$ -- depending on $\sigma$ and $K$ only -- and $M_2>0$ depending on $\sigma$, $K$, and $M_1$ -- such that
\[
F_*(\mc B_{BV^1,M_1})\subset \mc B_{BV^1,M_1}
\]
and
\[
F_*(\mc B_{BV^1,M_1}\cap\mc B_{BV^2,M_2})\subset\mc B_{BV^1,M_1}\cap\mc B_{BV^2,M_2}.
\]
\end{lemma}
\begin{proof}Pick $\phi$ s.t. $|\phi|_{BV^1}<\infty$. The following Lasota-Yorke inequality holds 
\begin{equation}\label{Eq:LY1}
|F_*\phi|_{BV^1}\le \lambda_1|\phi|_{BV^1}+D\|\phi\|_{L^1}
\end{equation}
with $\lambda_{1}\in[0,1)$ and $R_1>0$. In fact, for any $g\in \mc C^{1}(\mb T, \mb R)$ with $\Vert g\Vert_{\infty} \le 1$, by definition
\[
|F_*\phi|_{BV^1}\le \int_{\mb T}g'\, (F_*\phi)= \int_\T g'\circ F \,\phi.
\]
We use the identity
\[
\left(\frac{g\circ F}{F'}\right)'=\frac{(g\circ F)'}{F'}-(g\circ F)\frac{F''}{(F')^{2}}
\]
to obtain 
\begin{align*}
|F_*\phi|_{BV^1}  & \le \int_{\mb T} \, \left(\frac{g\circ F}{F'}\right)' \varphi + \int_{\mb T}  \, (g\circ F)\frac{F''}{(F')^{2}} \, \varphi\\
& \le \sigma^{-1} \int_{\mb T} \, \tilde{g} \varphi + D \Vert \varphi\Vert_{L^1}\\
& = \lambda_{1} |\phi|_{BV^1} + D \Vert \varphi\Vert_{L^1}
\end{align*}
where $\tilde{g}:=\sigma\frac{g\circ F}{F'}$ with $\Vert \tilde{g}\Vert_{\infty}\le 1$ and in the second term we used $\Vert g\Vert_{\infty}\le 1.$ Note that $\lambda_{1}=\sigma^{-1}$, $D=\max \frac{|F''|}{|(F')^{2}|}.$

By definition of the seminorm $|\cdot|_{BV^2}$,
\begin{align*}
|F_*\phi|_{BV^2}&\le \int_\T g''(F_*\phi)= \int_\T g''\circ F \,\phi.
\end{align*}
for any $g\in \mc C^2(\T,\R)$ with $\Vert g \Vert_\infty\le 1$.

Notice:
\begin{align*}
g''\circ F&=\frac{(g'\circ F)'}{F'}=\left(\frac{[g'\circ F]}{F'}\right)'+\frac{g'\circ F}{(F')^2}F''\\
&=\left(\frac{[(\frac{g\circ F}{F'})'+\frac{g\circ F}{(F')^2}F'']}{F'}\right)'+\frac{(g\circ F)'}{(F')^3}F''\\
&=\left(\frac{(\frac{g\circ F}{F'})'}{F'}\right)'+\left(\frac{\frac{g\circ F}{(F')^2}}{F'}F''\right)'+\left(\frac{g\circ F}{(F')^3}F''\right)'+3\frac{g\circ F}{(F')^4}F''-\frac{g\circ F}{(F')^3}F'''\\
&=\left(\frac{\frac{g\circ F}{F'}}{F'}\right)''+\left(\frac{\frac{g\circ F}{F'}}{(F')^2}F''\right)'+\left(\frac{\frac{g\circ F}{F'}}{(F')^2}F''\right)'+\left(\frac{g\circ F}{(F')^3}F''\right)'+3\frac{g\circ F}{(F')^4}F''-\frac{g\circ F}{(F')^3}F'''\\
&=\left(\frac{g\circ F}{(F')^2}\right)''+3\left(\frac{{g\circ F}}{(F')^3}F''\right)'+3\frac{g\circ F}{(F')^4}F''-\frac{g\circ F}{(F')^3}F'''
\end{align*}

The above implies that 
\begin{equation}\label{Eq:LY2}
|F_*\phi|_{BV^2}\le \lambda_2 |\phi|_{BV^2}+R_2|\phi|_{BV^1}+R_3\|\phi\|_{L^1}
\end{equation}
where $\lambda_2\in[0,1)$, $R_2>0$, and $R_3>0$ depend only on $\sigma$ and $K$. 

To conclude, one can first infer from \eqref{Eq:LY1} the existence of $\bar M_1>0$ such that $\mc B_{BV^1,M_1}$ is invariant for any $M_1>\bar M_1$, and for every $M_1>0$  the existence of $\bar M_2>0$ such that $\mc B_{BV^1,M_1}\cap\mc B_{BV^2,M_2}$ is invariant for any $M_2>\bar M_2$, and this implies the statement of the lemma.
\end{proof}

\begin{lemma}\label{Lem:RelationL1andBV1}
Let's consider $\varphi_{z},\varphi_{\bar{z}}\in L^{1}(\mathbb{T})$, then 
\begin{equation}\label{Eq:EstL1wrtBV1}
\|\phi_z-\phi_{\bar z}\|_{L^1}\le 2|\phi_z-\phi_{\bar z}|_{BV^1},
\end{equation}
so in particular, $\|\phi_z-\phi_{\bar z}\|_{BV^1}\le 3|\phi_z-\phi_{\bar z}|_{BV^1}$.
\end{lemma}
\begin{proof}
First, we claim that
\[
\{g\in \mc C^0(\T,\R):\, | g |_\infty\le1,\,\int_\T g=0\}\subset \{h':\,h\in \mc C^1(\T,\R),\,|h|_\infty\le 1\};
\]
pick any $g$ in the first set,  and define $h(x)=\int_0^xg(s)ds$. It is immediate to check that for $x\neq 0$, $h'(x)=g(x)$ and that $h'(0^+)=g(0)$; furthermore
\[
h'(0^-)=\lim_{h\rightarrow 0}\frac{1}{h}\int_0^{1-h}g(s)ds=\lim_{h\rightarrow 0}\frac{\int_0^{1}g(s)ds-\int_{1-h}^0g(s)ds}{h}=\lim_{h\rightarrow 0}\frac{-\int_{1-h}^0g(s)ds}{h}=g(0).
\] 
So, $h'$ is continuous in $0$. Now we prove  inequality \eqref{Eq:EstL1wrtBV1}. Since $\phi_z,\,\phi_{\bar z}$ are both probability densities
\begin{align*}
\|\phi_z-\phi_{\bar z}\|_{L^1}& =\sup_{| g |_\infty\le 1}\int_\T g(\phi_z-\phi_{\bar z}) \\
& = \sup_{| g |_\infty\le 1}\int_\T (g-\int g)(\phi_z-\phi_{\bar z})\\
& \le  \sup_{\substack {| g |_\infty\le 2\\\int g=0}}\int_\T g(\phi_z-\phi_{\bar z})\\
& \le \sup_{|h|_{\infty}\le 2}\int_\T h'(\phi_z-\phi_{\bar z})\\
& \le 2|\phi_z-\phi_{\bar z}|_{BV^1}
\end{align*}

where we used that $|g-\int g|_{\infty}\le 2$. 
\end{proof}

\begin{lemma}\label{VariationOfDensity}
    Suppose that  $h\in\mathcal{C}^{k+1}(\mathbb{T}\times\mathbb{T},\mathbb{R})$ and $f\in\mathcal{C}^{k}(\mathbb{T},\mathbb{T})$. Consider $\nu,\mu\in \mc M_{1,\Leb}([0,1]\times \T)$ absolutely continuous. Then for a.e. $z\in[0,1]$ the following holds
    $$\Vert F_{\nu,z}-F_{\mu,z} \Vert_{\mc{C}^k} \leq \alpha \, \Vert W\Vert_{L^{\infty}([0,1],L^1([0,1],\R))} \Vert h\Vert_{\mc{C}^{k+1}} \mathop{\mathrm{ess~sup}}\limits_{z'\in [0,1]}\Vert \nu_{z'}-\mu_{z'}\Vert_{W^1}.$$
    If $\nu,\,\mu$ are absolutely continuous and $\phi,\,\psi$ denote their densities, the above can be bounded as
    $$\Vert F_{\nu,z}-F_{\mu,z} \Vert_{\mc{C}^k} \leq \alpha \, \Vert W\Vert_{L^{\infty}([0,1],L^1([0,1],\R))} \Vert h\Vert_{\mc{C}^{k+1}} \mathop{\mathrm{ess~sup}}\limits_{z'\in [0,1]}\sup_{x\in\T}|\phi_{z'}(x)-\psi_{z'}(x)|.$$
\end{lemma}
\begin{proof}
    Recall that the $k$-th derivative of the dynamics in the fiber $\{z\}\times\mathbb{T}$ is given by
\begin{equation}\label{iderivative}
F^{(k)}_{\nu,z}(x)=f^{(k)}(x)+\alpha \int_{0}^{1}\int_{\mathbb{T}}dz' \, \, W(z,z') \partial_{1}^{(k)}h(x,y)d\nu_{z'}(y)
\end{equation}
which implies that for any $x\in\mathbb{T}$ and every $z\in [0,1]$ such that $W(z,\cdot)\in L^{1}(\T)$
\begin{align*}\label{estimativafuncionesenfibras}
   \left\vert F^{(k)}_{\nu,z}(x)-F^{(k)}_{\mu,z}(x) \right\vert & = \left\vert \alpha \int_{0}^{1}\int_{\mathbb{T}}\, dz' \, \, W(z,z') \, \partial_{1}^{(k)} h(x,y) \, d(\nu_{z'}(y)-\mu_{z'}(y)) \right\vert  \nonumber \\
         & \leq \alpha\,\Vert W(z,\cdot)\Vert_{L^1}  \, \Vert \partial^{(k+1)}_{1} h(x,\cdot) \Vert_{\infty} \Vert W(z,\cdot)\Vert_{L^1} \mathop{\mathrm{ess~sup}}\limits_{z'\in[0,1]}\Vert \nu_{z'}(\cdot)-\mu_{z'}(\cdot)\Vert_{W^1} \nonumber \\
         & \leq \alpha \, \Vert h\Vert_{\mc C^{k+1}} \Vert W(z,\cdot)\Vert_{L^1}\mathop{\mathrm{ess~sup}}\limits_{z'\in[0,1]}\Vert \nu_{z'}(\cdot)-\mu_{z'}(\cdot)\Vert_{W^1} \nonumber\\
         & \leq \alpha \, \Vert h\Vert_{\mc C^{k+1}} \| W\|_{L^{\infty}([0,1],L^1([0,1],\R))} \mathop{\mathrm{ess~sup}}\limits_{z'\in[0,1]}\Vert \nu_{z'}-\mu_{z'}\Vert_{W^1}
\end{align*}

The second statement follows from the first noticing that for any measures $\nu_{z'},\,\mu_{z'}$ with densities $\phi_{z'},\,\psi_{z'}$ one has $\|\nu_{z'}-\mu_{z'}\|_{W^1}\le \sup_{x\in \T}|\phi_{z'}(x)-\psi_{z'}(x)|$. 
\end{proof}

\begin{proposition}\label{psig}
Let $\psi \in L^1(\mathbb{T})$. Then,
\[
\sup_{\substack{\mathrm{Lip}(h) \leq 1 \\ \|h\|_\infty \leq 1}}  \int_{\mathbb{T}} h(s)\psi(s)\,ds 
\leq 
\sup_{\substack{g \in \mc C^1(\mathbb{T}) \\ \|g\|_\infty \leq 1}} \int_{\mathbb{T}} g'(s)\psi(s)\,ds.
\]
\end{proposition}
\begin{proof}
Let $h \in \mathrm{Lip}(\mathbb{T})$ with $\|h\|_\infty \le 1$ and $\mathrm{Lip}(h) \le 1$. Since Lipschitz functions on the torus can be uniformly approximated by $\mc C^1(\mb T)$ functions with controlled derivatives, there exists a sequence $\{g_n\} \subset \mc C^1(\mathbb{T})$ such that
\begin{itemize}
    \item $g_n \in \mc C^\infty(\mathbb{T}) \subset \mc C^1(\mathbb{T})$,
    \item $\|g_n\|_\infty \leq \|h\|_\infty \leq 1$,
    \item $\|g_n'\|_\infty \leq \mathrm{Lip}(h) \leq 1$,
    \item $g_n \to h$ uniformly as $n \to \infty$.
\end{itemize}

Since $\psi \in L^1(\mathbb{T})$ and $g_n \to h$ uniformly, by dominated convergence:
\[
\lim_{n\to\infty}\int_{\mathbb{T}} g_n(s)\psi(s)\,ds =\int_{\mathbb{T}} h(s)\psi(s)\,ds.
\]
\end{proof}

\section{Projective cones}

For completeness, we recall some facts on projective cones and the Hilbert metric. We follow \cite{viana1997stochastic}.
Let $E$ be a vector space, any subset $\mathcal{V}\subset E\setminus \{0\}$ is a cone if $v\in \mathcal{V}$ and $t>0$ then $tv\in\mathcal{V}.$
The cone $\mathcal{V}$ is convex if
\begin{align*}
    v_{1}, v_{2} \in \mathcal{V} \ \ \mbox{and} \ \ t_{1},t_{2}>0 \ \ \mbox{then} \ \ t_{1}v_{1}+t_{2}v_{2}\in \mathcal{V}.
\end{align*}
The closure $\overline{\mathcal{V}}$ of a cone $\mathcal{V}$ is defined as
\begin{align*}
    \overline{\mathcal{V}}:=\{ u\in E : \ \exists v\in \mathcal{V} \ \mbox{and} \ (t_{n})_{n\in\mathbb{N}} \searrow 0 \ \mbox{such that} \ u+t_{n}v \in \mathcal{V} \ \mbox{for all} \ n \geq 1 \}.
\end{align*}
We assume that the cone $\mathcal{V}$ satisfies $
    \overline{\mathcal{V}}\cap (- \overline{\mathcal{V}})=\{0\}.$
Consider $v_{1}, v_{2} \in \mathcal{V}$ and define 
    $$\alpha(v_{1},v_{2}):=\sup \{ t>0: v_{2}-tv_{1}\in \mathcal{V}\}; \qquad
    \beta(v_{1},v_{2}):=\inf \{ s>0: sv_{1}-v_{2}\in \mathcal{V}\},$$
with $\sup \emptyset =0$ and $\inf \emptyset=+\infty$.
The Hilbert metric is defined as 
$\theta:\mathcal{V}\times\mathcal{V}\to \mathbb{R}\cup \{\infty\}$ by
\begin{equation*}
    \theta(v_{1},v_{2}) = 
\begin{cases} 
    \log \dfrac{\beta(v_{1},v_{2})}{\alpha(v_{1},v_{2})} & \text{if } \alpha >0  \text{ and } \beta<+\infty \\ 
    +\infty & \text{if } \alpha=0 \text{ or } \beta=+\infty.  
\end{cases}
\end{equation*}

Let $X$ be a compact metric space and consider $a>0$, $\nu\in (0,1]$ and define
    $$\mathcal{C}=\mathcal{C}(a,\nu):= \{ \varphi\in E : \varphi >0 \ \text{and} \ \log \varphi \ \text{is} \ (a, \nu)\text{-Hölder}\}$$
    the condition that $\log \varphi$ is $(a,\nu)$-Hölder means that, for all $x,y \in X$
    $$\exp(-ad(x,y)^{\nu})\leq \dfrac{\varphi(x)}{\varphi(y)}\leq \exp(ad(x,y)^{\nu}).$$
    In particular, if we consider $E=\mathcal{C}(\mathbb{T})$, $\nu=1$ and $d(x,y)=\vert x-y\vert$ denotes the Euclidean distance on $\mathbb{T}$ between $x$ and $y$ we  we obtain the usual cone of log-Lipschitz functions 
    $$\mathcal{V}_{a}:=\mathcal{C}(a,1)=\left\{\psi \in \mathcal{C}^0(\mathbb{T},\R^): \frac{\psi(x)}{\psi(y)}\le e^{a|x-y|} \right\}.$$

\begin{lemma}\label{C0ANDhm}
Suppose $\phi,\,\psi\in \mc V_a$ are probability density, then there is a constant $C>0$ --  depending on $a>0$ only -- such that 
\[
\sup_{x\in\T}|\phi(x)-\psi(x)|\le C\theta_a(\phi,\psi).
\]
\end{lemma}

\begin{proof}
To rigorously establish this, we must follow these steps: 

\textit{Step 1.} Notice that 
\begin{equation}\label{RATIOdensities}
1\leq \sup_{x\in\mathbb{T}} \dfrac{\psi(x)}{\varphi(x)};
\end{equation}
in fact,  if $\sup_{x\in\mathbb{T}} \frac{\psi(x)}{\varphi(x)}<1,$ then for all $x\in \mathbb{T}$ we have that $\psi(x)<\varphi(x)$ which means they cannot be both probability density leading to a contradiction.

\textit{Step 2.} Given $x_{0}>0$ there is a constant depending on $x_{0}$ that is $C=C(x_{0})$, such that if $x<x_{0}$, then $x<C(x_{0})\log(1+x).$  In fact, let us consider the function 
$$l(x)=\log(x+1)-\dfrac{x}{x+1}$$
$l$ is increasing and $l(0)=0$ then $l(x) \geq l(0)=0$, thus,
$
\dfrac{x}{x+1}\leq \log(x+1)
$
and $x<(x_0+1)\log(x+1)$.

\textit{Step 3.} Since  $\varphi,\psi\in  \mathcal{V}_{a}$ we know that $\varphi(y)\leq \varphi(x)e^{a\vert x-y\vert}$ and $\psi(y)\leq\psi(x)e^{a\vert x-y\vert}$ for all $x,y \in \mathbb{T}.$ So, taking the supremum and infimum of the inequality, we obtain that
$$\sup_{y}\varphi(y)\leq e^{\frac{a}{2}} \inf_{x}\varphi(x) \ \ \ \text{and} \ \ \ \sup_{y}\psi(y)\leq e^{\frac{a}{2}} \inf_{x}\psi(x)$$
which implies, together with the fact that  $\varphi,\psi$ are densities, 
$$e^{-\frac{a}{2}} = e^{-\frac{a}{2}} \int_{\mathbb{T}} \varphi(y) dy \leq  e^{-\frac{a}{2}} \sup_{y}\varphi(y)\leq \inf_{x}\varphi(x)\leq \int_{\mathbb{T}} \varphi(x)dx = 1.$$
Analogously for the density $\psi,$ we obtain that
$$e^{-\frac{a}{2}}\leq \inf_{x}\varphi(x)\leq 1 \ \ \ \text{and} \ \ \ e^{-\frac{a}{2}}\leq \inf_{x}\psi(x)\leq 1$$
A similar argument to the above also shows that
$$1 \leq \sup_{x}\varphi(x)\leq e^{\frac{a}{2}} \ \ \ \text{and} \ \ \ 1 \leq \sup_{x}\psi(x)\leq e^{\frac{a}{2}}.$$
Thus, joining both inequalities, we obtain for all $x\in\mathbb{T}$ 
\begin{equation}
e^{-a}\leq \dfrac{\psi(x)}{\varphi(x)} \leq e^{a},
\end{equation}
in particular the ratio of the densities is upper bounded by some number that depends on the cone $\mathcal{V}_{a}$ and initially we saw that the ratio of the densities has to be greater than or equal to 1, so by \textit{Step 2}, there exists $C>0$ such that 
\begin{equation}\label{Eq:Rat-1Est}
\dfrac{\psi(x)}{\varphi(x)}-1\leq C\log \left(\dfrac{\psi(x)}{\varphi(x)}\right).
\end{equation}
Finally, inequality \eqref{RATIOdensities} implies that
\begin{equation}
1\leq \sup_{x\in \T} \dfrac{\psi(x)}{\varphi(x)} \leq \beta(\varphi,\psi),
\end{equation} 
and
$$\dfrac{\psi(x)}{\varphi(x)}\leq \beta(\varphi,\psi)\beta(\psi,\varphi)$$
then from \eqref{Eq:Rat-1Est}
\begin{equation}\label{finalexpre}
\dfrac{\psi(x)}{\varphi(x)}-1\leq C \, \theta_{a}(\varphi,\psi)
\end{equation}
and
\begin{equation}\label{C0andHILBERT}
 \vert\psi(x)-\varphi(x) \vert\leq C'\, \theta_{a}(\varphi,\psi).
\end{equation}
\end{proof}
\end{appendices}

\bibliographystyle{abbrv}

\bibliography{bibliography}

\end{document}